\documentclass[10pt]{article}

\usepackage[a4paper,margin=1in]{geometry}
\usepackage{amsmath,amssymb,amsthm,mathtools}
\usepackage{microtype}
\usepackage[skip=0.45\baselineskip plus 2pt,indent=0pt]{parskip}
\usepackage{float}
\usepackage{tikz}
\usepackage{scalerel}
\usepackage[colorlinks=true,linkcolor=blue,citecolor=blue,urlcolor=blue]{hyperref}

\numberwithin{equation}{section}

\newtheorem{theorem}{Theorem}[section]
\newtheorem{lemma}[theorem]{Lemma}
\newtheorem{proposition}[theorem]{Proposition}
\newtheorem{corollary}[theorem]{Corollary}

\newtheorem{fact}[theorem]{Fact}
\newtheorem{definition}[theorem]{Definition}

\DeclareMathOperator{\op}{op}
\DeclareMathOperator{\rev}{rev}

\newcommand{\R}{\mathbb{R}}

\newcommand{\N}{\mathbb{N}}

\newsavebox\vvvbox
\savebox\vvvbox{\tikz{
    \draw[black,fill=black] (90:1) circle (.35);
    \draw[black,fill=black] (210:1) circle (.35);
    \draw[black,fill=black] (330:1) circle (.35);
    \draw[opacity=0] (0:1.2) circle (0.1);
}}
\newcommand{\vvv}{\mathord{\scaleobj{1.2}{\scalerel*{\usebox{\vvvbox}}{x}}}}
\newcommand{\pivvv}{\pi_{\vvv}}

\newcommand{\1}{\mathbf{1}}
\newcommand{\cF}{\mathcal F}
\newcommand{\cL}{\mathcal L}
\newcommand{\cM}{\mathcal M}
\newcommand{\PiU}{\Pi_{\vvv,\infty}}

\title{\Large\bf Intervals of uniform Tur\'an densities}
\author{%
Heng~Li\thanks{\scriptsize School of Mathematics, Shandong University, Jinan, China, and Extremal Combinatorics and Probability Group, Institute for Basic Science, Daejeon, South Korea. Email:~\texttt{heng.li@sdu.edu.cn}.}\and
Xizhi Liu\thanks{\scriptsize School of Mathematical Sciences, University of Science and Technology of China, Hefei, China.
    Email:~\texttt{liuxizhi@ustc.edu.cn}.}
\and
Oleg Pikhurko\thanks{\scriptsize Mathematics Institute and DIMAP, University of Warwick, Coventry CV4 7AL, UK.
    Email:~\texttt{pikhurko@gmail.com}.}}
\date{\today}

\begin{document}

\maketitle

\begin{abstract}
    We prove that the set $\PiU$ of uniform Tur\'an densities of possibly infinite families of $3$-graphs contains a terminal interval: there exists $\delta>0$ such that $[1-\delta,1]\subseteq\PiU$. Consequently, $\PiU$ has positive Lebesgue measure and Hausdorff dimension $1$.
\end{abstract}

%%%%%%%%%%%%%%%%%%%%%%%%%%%%%%%%%%%%%%%%%%%%%%%%%%%%
\section{Introduction}
\label{sec:introduction}
%%%%%%%%%%%%%%%%%%%%%%%%%%%%%%%%%%%%%%%%%%%%%%%%%%%%

Given $d\in[0,1]$ and $\eta>0$, a $3$-uniform hypergraph, or simply a $3$-graph, $H$ on $n$ vertices is called \emph{$(d,\eta,\vvv)$-dense} if for every subset $U\subseteq V(H)$ of vertices, the number of edges in the induced subgraph $H[U]$ satisfies $e_H(U)\ge d\binom{|U|}{3}-\eta n^3$. 
Thus the edge-density requirement is imposed simultaneously on every vertex subset of the host graph.
For a family $\cF$ of $3$-graphs, its \emph{uniform Tur\'an density} $\pivvv(\cF)$ is the supremum of the numbers $d$ such that, for every $\eta>0$ and $n_0\in\N$, there exists an $\cF$-free $(d,\eta,\vvv)$-dense $3$-graph on at least $n_0$ vertices.
We write
\[
\PiU
\coloneqq \{\pivvv(\cF)\colon \cF\text{ is a possibly infinite family of
    $3$-graphs}\}.
\]
We allow $\cF=\varnothing$, and hence $1\in\PiU$.

The uniformly dense Tur\'an problem was initiated by Erd\H{o}s and S\'os~\cite{ErdosSos1982}.
It belongs to a broader family of quasirandom Tur\'an problems in which density is tested on prescribed local configurations.
Determining exact uniform Tur\'an densities is difficult even for small forbidden $3$-graphs.
Glebov, Kr\'al' and Volec~\cite{GlebovKralVolec2016} proved that $\pivvv(K_4^{(3)-})=1/4$, where $K_4^{(3)-}$ is the tetrahedron with one edge removed.
Reiher, R\"odl and Schacht subsequently developed a systematic framework for such problems, including tetrahedron embeddings under vertex--pair density tests, weakly quasirandom $3$-graphs, and the more general $j$-uniform density conditions for $r$-graphs \cite{ReiherRodlSchacht2016,ReiherRodlSchacht2018,ReiherRodlSchacht2018Mantel}.
Their arguments use the regularity and counting lemmas from~\cite{RodlSchacht07a,RodlSchacht07b}, within the broader hypergraph regularity theory developed in~\cite{Gowers2006,RodlSkokan2004,NagleRodlSchacht2006}.
In particular, they re-proved that $\pivvv(K_4^{(3)-})=1/4$ by an independent regularity-based argument~\cite{ReiherRodlSchacht2018}.
Reiher, R\"odl and Schacht~\cite{ReiherRodlSchacht2018Vanishing} also characterized the $3$-graphs of vanishing uniform Tur\'an density and showed that the uniform Tur\'an density of every single $3$-graph is either $0$ or at least $1/27$.
Garbe, Kr\'al' and Lamaison~\cite{GarbeKralLamaison2024} proved that this first gap is sharp by constructing a $3$-graph with uniform Tur\'an density exactly $1/27$.
Buci\'c, Cooper, Kr\'al', Mohr and Munh\'a Correia~\cite{BucicCooperKralMohrMunhaCorreia2023} determined the uniform Tur\'an densities of all tight cycles of length at least $5$: the value is $0$ when the length is divisible by $3$, and $4/27$ otherwise.
Nevertheless, the uniform Tur\'an density of the tetrahedron $K_4^{(3)}$ remains unknown.
For broader accounts of uniformly dense and related restricted extremal problems, see the surveys of Reiher~\cite{Reiher2020} and Schacht~\cite{Schacht2022ICM}.

Beyond determining individual densities, one may ask about the global structure of the set of all attainable values.
There are parallel questions for several notions of Tur\'an density.
Ordinary Tur\'an density controls only the total edge count; its graph-theoretic foundations are the theorems of Tur\'an and Erd\H{o}s--Stone--Simonovits \cite{Turan1941,ErdosStone1946,ErdosSimonovits1966}, while the existence of the limiting density for hypergraphs follows from the averaging argument of Katona, Nemetz and Simonovits~\cite{KatonaNemetzSimonovits1964}.
See \cite{Furedi1991,Sidorenko1995,Keevash2011} for general background.
Degree-based variants instead require every fixed set of vertices to extend to sufficiently many edges.
These include the codegree and $\ell$-degree Tur\'an densities studied by Mubayi and Zhao~\cite{MubayiZhao07} and by Lo and Markstr\"om~\cite{LoMarkstrom14}.
In particular, Lo and Markstr\"om proved that, for $r>\ell>1$, the set of possible $\ell$-degree Tur\'an densities is dense in $[0,1)$.
The local gaps in such sets are described by the language of jumps.
For $X\subseteq[0,1]$, a point $\alpha\in[0,1)$ is a \emph{jump of $X$} if $X\cap(\alpha,\alpha+c)=\varnothing$ for some $c>0$, and is a \emph{non-jump of $X$} otherwise.
For the ordinary Tur\'an density of $r$-graphs, Erd\H{o}s~\cite{Erdos1964} proved that every $\alpha\in[0,r!/r^r)$ is a jump and conjectured that every $\alpha\in[0,1)$ is a jump.
Frankl and R\"odl~\cite{FranklRodl1984} disproved this conjecture by constructing non-jumps for every $r\ge3$.
Among later developments, Baber and Talbot~\cite{BaberTalbot2011} found the first explicit jumps above the Erd\H{o}s threshold, while Frankl, Peng, R\"odl and Talbot~\cite{FPRT2007} proved that $5r!/(2r^r)$ is a non-jump for every $r\ge3$.
Peng further developed Lagrangian constructions of non-jumping numbers and studied jumping densities, particularly in higher uniformities \cite{Peng07NonJumping4Uniform,Peng2007II,Peng2008I,Peng09JumpingDensities}.
More recently, Liu and Mubayi~\cite{LiuMubayi2026FourNinths} proved that $4/9$ is a non-jump for $3$-graphs.

For uniform Tur\'an density, palettes (defined in Section~\ref{se:palettes} here) have emerged as the principal framework for studying attainable values.
Lamaison~\cite{Lamaison2024Palettes} proved that palettes determine the uniform Tur\'an density of every single $3$-graph.
King, Sales and Sch\"ulke~\cite{KingSalesSchulke2024} obtained realization results for scaled $3$-graph Lagrangians, and King, Piga, Sales and Sch\"ulke~\cite{KingPigaSalesSchulke2025} proved that every finite-palette Lagrangian is the uniform Tur\'an density of a finite forbidden family.
Kr\'al', Ku\v{c}er\'ak, Lamaison and Tardos~\cite{KralKucerakLamaisonTardos2025} subsequently gave a palette-separation theorem in terms of palette homomorphisms.
Lin, Sun, Wang and Zhou~\cite{LinSunWangZhou2026} extended the palette characterization to arbitrary families and to the $(r-2)$-uniform Tur\'an density of $r$-graphs for every $r\ge3$.
These realization results also connect the jump phenomena in the ordinary and uniformly dense settings: every non-jump for ordinary $3$-graph Tur\'an densities yields a non-jump of $\PiU$, and the finite-palette theorem gives the analogous conclusion within the uniform Tur\'an densities of finite forbidden families~\cite{KingSalesSchulke2024,KingPigaSalesSchulke2025}.

A stronger global question is whether a set of attainable densities actually contains an interval.
In the ordinary setting, Pikhurko~\cite{Pikhurko2014} and Grosu~\cite{Grosu2016} initiated a topological study of Tur\'an density sets; Grosu asked, in particular, about their Hausdorff dimension.
Frankl, Peng, R\"odl and Talbot~\cite{FPRT2007} asked whether, for every $r\ge3$, the Tur\'an densities of possibly infinite families of $r$-graphs contain a terminal interval.
Liu and Pikhurko~\cite{LiuPikhurkoIntervals} recently answered this question affirmatively.
They also proved that the uniform Tur\'an densities of finite families of $3$-graphs are dense in a non-degenerate interval.
That density result does not by itself imply that $\PiU$ contains an interval, since it is not known whether $\PiU$ is closed.
The aim of the present paper is to obtain pointwise realization of an interval in the uniformly dense setting by allowing the forbidden family to be infinite.
Note that the use of possibly infinite forbidden families is necessary for an interval statement, since there are only countably many finite families of $3$-graphs.

Our main result is the following.

\begin{theorem}\label{thm:main}
    There exists a constant $\delta>0$ such that $[1-\delta,1]\subseteq \PiU$.
    In particular, $\PiU$ has positive Lebesgue measure and Hausdorff dimension $1$.
\end{theorem}

The proof has three steps.
First, rather than trying to prove that $\PiU$ is closed, we directly realize certain decreasing limits of finite-palette Lagrangians.
More precisely, suppose that we have a sequence of finite palettes, with a palette homomorphism from each palette to the preceding one, and that their Lagrangians decrease to a limit $x$.
We show that $x$ is the uniform Tur\'an density of a suitable, possibly infinite, forbidden family (Proposition~\ref{prop:projective-chain}).
The homomorphisms make the lower-bound constructions from different levels compatible, while palette separation supplies an obstruction to every palette that could force a density above $x$.

Second, we construct enough such chains to fill an interval.
Starting with a palette on a large regular graph, we repeatedly take spectrally controlled $2$-lifts using the Bilu--Linial signing theorem~\cite{BiluLinial2006} (Proposition~\ref{prop:lift-tower}).
The chosen lift supports the next palette, while the complementary lift records the repeated-color patterns (that is, ordered triples in which one color appears exactly twice) that are removed (Lemma~\ref{lem:exact-lift}).
A global cubic expander estimate shows that any imbalance between the two lifted copies of a color costs more than it can gain (Lemma~\ref{lem:cubic-expander}).
Summing this estimate over all levels shows that the uniform weighting remains the unique maximizer, even after deleting up to four prescribed tripartite perfect matchings at each step (Theorem~\ref{thm:tower-stability}).
These matchings are chosen persistently: every positive level supplies four fresh deletions without disturbing the choices available later.
Deleting $k_i\in\{0,1,2,3,4\}$ matchings at level $i$ acts like choosing a base-$4$ digit, since the number of colors doubles and the cost is proportional to $4^{-i}$ (Proposition~\ref{prop:branch-values}).
The redundant digit set $\{0,1,2,3,4\}$ fills a whole interval (Lemma~\ref{lem:interval-coding}).
Hence, by the first step, the branch limits form a non-degenerate interval $[a,b]\subseteq\PiU$ (Theorem~\ref{thm:base-interval}).

Third, the $M$-fold complete join, a standard technique, preserves palette homomorphisms and sends a Lagrangian $x$ to $1-(1-x)/M^2$ (Lemma~\ref{lem:join-formula}).
For all sufficiently large consecutive values of $M$, the images of $[a,b]$ overlap and approach $1$, producing the required terminal interval (Theorem~\ref{thm:main}).

% The proof uses the following three previously known results: the Bilu--Linial signing theorem~\cite{BiluLinial2006}, the palette separation theorem of Kr\'al', Ku\v{c}er\'ak, Lamaison and Tardos~\cite{KralKucerakLamaisonTardos2025}, and the arbitrary-family palette characterization of Lin, Sun, Wang and Zhou~\cite{LinSunWangZhou2026} (which extends that of the single $3$-graph version of Lamaison~\cite{Lamaison2024Palettes}).
% We state below the precise forms that we use; all other ingredients are proved in the paper.

The paper is organized as follows.
Section~\ref{sec:preliminaries} develops the palette formalism and constructs the spectrally controlled graph tower.
Section~\ref{sec:projective-limits} proves the projective-chain realization principle.
Section~\ref{sec:palette-lift} places palettes on the graph tower and establishes their multiscale stability.
Section~\ref{sec:interval-construction} completes the interval construction. 
Section~\ref{sec:concluding-remarks} records the extension to all uniformities.

%%%%%%%%%%%%%%%%%%%%%%%%%%%%%%%%%%%%%%%%%%%%%%%%%%%%
\section{Preliminaries}
\label{sec:preliminaries}
%%%%%%%%%%%%%%%%%%%%%%%%%%%%%%%%%%%%%%%%%%%%%%%%%%%%

For an integer $t\ge0$, write $[t]\coloneqq \{1,\ldots,t\}$, with $[0]=\varnothing$.
All graphs, hypergraphs and palettes considered below are finite unless explicitly stated otherwise.
We write $a_m\downarrow a$ if the sequence $(a_m)$ is strictly decreasing and converges to $a$.
Unless a different subscript is displayed, all vector norms are Euclidean.
For a real matrix $A$, its \emph{Euclidean operator norm} is
\[
\|A\|_{\op}\coloneqq \max_{\|\mathbf{x}\|_2=1}\|A\mathbf{x}\|_2.
\]
If $A$ is symmetric, then $\|A\|_{\op}$ is the largest absolute value of an eigenvalue of $A$.
Every supremum below is taken in the ordered interval $[0,1]$; in particular, the supremum of the empty set is $0$. 

Let $J_V$ denote the square matrix indexed by a set $V$ with all entries $1$; its eigenvalues are $|V|$ (simple) and $0$ (of multiplicity $|V|-1$). We may write $J$ when the index set $V$ is understood.

\subsection{Uniform Tur\'an density and palettes}\label{se:palettes}

If $H$ is a $3$-graph and $U\subseteq V(H)$, write $e_H(U)$ for the number of edges of $H$ contained in $U$.
A $3$-graph is \emph{$\cF$-free} if it contains no member of $\cF$ as a subgraph.
We record the quantifiers in the definition of uniform Tur\'an density for later reference.

\begin{definition}\label{def:uniform-density}
    For $d\in[0,1]$ and $\eta>0$, a $3$-graph $H$ on $n$ vertices is \emph{$(d,\eta,\vvv)$-dense}, or \emph{uniformly $(d,\eta)$-dense}, if $e_H(U)\ge d\binom{|U|}{3}-\eta n^3$ for every $U\subseteq V(H)$.
    
    For a family $\cF$ of $3$-graphs, let $A(\cF)$ be the set of all $d\in[0,1]$ with the following property: for every $\eta>0$ and $n_0\in\N$, there exists an $\cF$-free, uniformly $(d,\eta)$-dense $3$-graph on at least $n_0$ vertices.
    The \emph{uniform Tur\'an density} of $\cF$ is $\pivvv(\cF)\coloneqq \sup A(\cF)$, where we use the convention $\sup\varnothing=0$ in the ordered interval $[0,1]$.
    % We allow $\cF=\varnothing$, so $\pivvv(\varnothing)=1$.
    % Thus
    % \[
    %   \PiU=\{\pivvv(\cF)\colon\cF\text{ is a possibly infinite family of
        %   $3$-graphs}\}.
    % \]
\end{definition}

A \emph{palette} is a pair $P=(C,T)$, where $C$ is a finite nonempty set of colors and $T\subseteq C^3$ is a set of ordered triples of colors which we call \emph{admissible} or \emph{allowed}.
We call an ordered triple in $C^3$ a \emph{pattern}; its three coordinates record, in a prescribed order, the colors assigned to the three pairs of a vertex triple, and $T$ is the set of patterns allowed by $P$.
A \emph{weighting} of $C$ is a vector $\mathbf{x}=(x_c)_{c\in C}$ in the simplex
\[
\Delta_C\coloneqq \Big\{\mathbf{x}\in[0,1]^C\colon\sum_{c\in C}x_c=1\Big\}.
\]
The \emph{uniform weighting} on $C$ is $\mathbf{u}_C\coloneqq |C|^{-1}\1_C$, where $\1_C$ is the all-1 vector indexed by $C$.
When the underlying color set is clear and has size $N$, we may write $\mathbf{u}_N$ for $\mathbf{u}_C$ and $\1$ for $\1_C$.
The \emph{palette polynomial} and \emph{palette Lagrangian} are
\[
\Lambda_P(\mathbf{x})\coloneqq \sum_{(a,b,c)\in T}x_ax_bx_c,
\quad\text{and}\quad 
\lambda(P)\coloneqq \sup_{\mathbf{x}\in\Delta_C}\Lambda_P(\mathbf{x}),
\]
where the supremum is attained by compactness. The \emph{unweighted density} is $d(P)\coloneqq |T|/|C|^3=\Lambda_P(\mathbf{u}_C)$.

Let $F$ be a $3$-graph.
We say that $F$ is \emph{$P$-colorable} if there are a linear order $<$ on $V(F)$ and a coloring $\phi\colon\binom{V(F)}2\to C$ such that $\bigl(\phi(uv),\phi(uw),\phi(vw)\bigr)\in T$ for every edge $u<v<w$ of $F$.
A family $\cF$ is \emph{$P$-colorable} if at least one of its members is $P$-colorable; equivalently, $\cF$ is not $P$-colorable precisely when no member of $\cF$ is $P$-colorable.

A \emph{homomorphism} $f\colon P\to Q$ is a map on color sets such that admissible triples of $P$ map to admissible triples of $Q$.
If $f\colon C\to C'$ is a map between finite color sets and $\mathbf{x}\in\Delta_C$, the \emph{push-forward} of $\mathbf{x}$ under $f$ is the weighting $f_*\mathbf{x}\in\Delta_{C'}$ defined by $(f_*\mathbf{x})_{c'}\coloneqq \sum_{c\in f^{-1}(c')}x_c$ for $c'\in C$. 
Thus the weight of $c'$ is the total weight of the colors mapped to $c'$.
The \emph{reverse palette} $\rev(P)$ has the same color set, with $(a,b,c)\in\rev(P)$ if and only if $(c,b,a)\in P$.
An \emph{isomorphism of palettes} is a bijective homomorphism whose inverse is also a homomorphism; equivalently, it is a relabeling of the color set that preserves admissible and inadmissible triples.
Reversing a witnessing vertex order shows that a $3$-graph is $P$-colorable if and only if it is $\rev(P)$-colorable.
The definition also gives $d(\rev(P))=d(P)$ and $\lambda(\rev(P))=\lambda(P)$.

The following elementary consequences of the definitions allow us to pass colorability and Lagrangian bounds along palette homomorphisms.

\begin{fact}\label{fact:palette-monotonicity}
    Let $P$ and $Q$ be finite palettes.
    \begin{enumerate}
        \item If a $3$-graph $F$ is $P$-colorable, then every subgraph of $F$ is $P$-colorable.
        \item If $P\to Q$, then every $P$-colorable $3$-graph is $Q$-colorable.
        \item If $P\to Q$, then $\lambda(P)\le\lambda(Q)$.
    \end{enumerate}
\end{fact}

We first record an elementary probabilistic construction and then state two known structural results about palettes.
The probabilistic statement is the standard weighted random-palette construction in the convention used here.
This is the usual independent pair-coloring argument; compare the discussion preceding 
%King, Sales and Sch\"ulke~
\cite[Lemma~3.4]{KingSalesSchulke2024}, whose convention interchanges our second and third palette coordinates.
The resulting existence statement is standard.

\begin{lemma}\label{lem:weighted-palette}
    For every finite palette $P$, every $\varepsilon,\eta>0$, and every $n_0\in\N$, there exists a $P$-colorable $3$-graph $H$ on at least $n_0$ vertices that is uniformly $\bigl(\max\{\lambda(P)-\varepsilon,0\},\eta\bigr)$-dense.
\end{lemma}

\begin{proof}
    Write $P=(C,T)$ and choose a weighting $\mathbf{x}\in\Delta_C$ with $\Lambda_P(\mathbf{x})=\lambda(P)$.
    Let $n\ge n_0$ be sufficiently large.
    Independently for every pair of vertices of $[n]$, assign color $c$ with probability $x_c$.
    Let $H$ consist of those triples $u<v<w$ for which $\bigl(\phi(uv),\phi(uw),\phi(vw)\bigr)\in T$.
    The natural order and the sampled pair coloring witness that $H$ is $P$-colorable.
    
    Put $d\coloneqq \max\{\lambda(P)-\varepsilon,0\}$.
    For a fixed $U\subseteq[n]$, let $X_U=e_H(U)$.
    Every triple of $U$ is present with probability $\Lambda_P(\mathbf{x})=\lambda(P)$, and hence $\mathbb E X_U=\lambda(P)\binom{|U|}{3}$. 
    Changing the color assigned to a single pair affects only triples containing that pair, so it changes $X_U$ by at most $n-2<n$.
    If $d\binom{|U|}{3}-\eta n^3<0$, the required lower bound is automatic.
    Otherwise the bad event implies $X_U-\mathbb E X_U< -\eta n^3$ because $d\le\lambda(P)$.
    McDiarmid's bounded-differences inequality~\cite{Mcdiarmid89} and the bound $\sum_{e\in\binom{[n]}2} n^2 =\binom n2n^2<\frac{n^4}{2}$ therefore give
    \[
    \Pr\left(X_U<d\binom{|U|}{3}-\eta n^3\right)
    \le \exp\!\left(-\frac{2\eta^2n^6}{\frac{n^4}{2}}\right)
    =\exp\!\left(-4\eta^2n^2\right).
    \]
    The union bound over all $2^n$ choices of $U$ shows that the probability of any failure is at most $2^n\exp(-4\eta^2n^2)$, which is less than $1$ for all sufficiently large $n$.
    Thus some sampled coloring produces the required $H$.
\end{proof}

We now state two palette-classification results.
The first is the single-palette separation theorem of Kr\'al', Ku\v{c}er\'ak, Lamaison and Tardos~\cite[Theorem~13]{KralKucerakLamaisonTardos2025}.
Their inverse palette $\operatorname{inv}(Q)$ is exactly $\rev(Q)$, and their ordering of the three pair colors agrees with ours.

\begin{theorem}[{\cite[Theorem~13]{KralKucerakLamaisonTardos2025}}]
    \label{thm:palette-separation}
    For finite palettes $P$ and $Q$, there exists a finite $3$-graph that is $P$-colorable but not $Q$-colorable if and only if there is no homomorphism $P\to Q$ and no homomorphism $P\to\rev(Q)$.
\end{theorem}

The second result originates with Lamaison~\cite{Lamaison2024Palettes}, who proved the palette characterization for a single forbidden $3$-graph.
We use the extension to an arbitrary, possibly infinite, family of $3$-graphs given by Lin, Sun, Wang and Zhou~\cite[Theorem~1.4]{LinSunWangZhou2026}. Their proof is inspired by Lamaison's strategy but requires an additional compactness argument.
For $u<v<w$, their coordinate order is $\bigl(\phi(vw),\phi(uw),\phi(uv)\bigr)$, the reverse of ours.
Applying their theorem to the reverse palette, and using the invariance observations above, gives the following equivalent formulation.

\begin{theorem}[\cite{Lamaison2024Palettes}, {\cite[Theorem~1.4]{LinSunWangZhou2026}}]
    \label{thm:palette-characterization}
    For every family $\cF$ of $3$-graphs,
    \[
    \pivvv(\cF)
    =\sup\{d(Q)\colon\text{no member of $\cF$ is $Q$-colorable}\}.
    \]
\end{theorem}

%%%%%%%%%%%%%%%%%%%%%%%%%%%%%%%%%%%%%%%%%%%%%%%%%%%%
\subsection{A spectrally controlled tower of \texorpdfstring{$2$}{2}-lifts}
\label{sec:lift-tower}
%%%%%%%%%%%%%%%%%%%%%%%%%%%%%%%%%%%%%%%%%%%%%%%%%%%%

The purpose of this subsection is to prepare the graph sequence that will later carry the palette construction.
We first recall the $2$-lift operation and its elementary spectral decomposition.

\begin{figure}[H]
    \centering
    \begin{tikzpicture}[
        lift vertex/.style={
            circle,draw=black,fill=black!75,inner sep=1.3pt
        },
        lift edge/.style={draw=black,line width=1pt},
        panel label/.style={font=\small}
        ]
        % The base edge.
        \node[lift vertex,label=left:$u$]  (base-u) at (0,0) {};
        \node[lift vertex,label=right:$v$] (base-v) at (1.25,0) {};
        \draw[lift edge] (base-u) -- (base-v);
        % \node[panel label] at (0.625,-0.9) {$uv$};
        
        % The parallel matching.
        \node[lift vertex,label=left:$u^+$]  (par-u-plus)  at (3.45,0.45) {};
        \node[lift vertex,label=left:$u^-$]  (par-u-minus) at (3.45,-0.45) {};
        \node[lift vertex,label=right:$v^+$] (par-v-plus)  at (4.8,0.45) {};
        \node[lift vertex,label=right:$v^-$] (par-v-minus) at (4.8,-0.45) {};
        \draw[lift edge] (par-u-plus) -- (par-v-plus);
        \draw[lift edge] (par-u-minus) -- (par-v-minus);
        % \node[panel label] at (4.125,-0.9) {$s(uv)=+1$};
        
        % The crossed matching.
        \node[lift vertex,label=left:$u^+$]  (cross-u-plus)  at (6.95,0.45) {};
        \node[lift vertex,label=left:$u^-$]  (cross-u-minus) at (6.95,-0.45) {};
        \node[lift vertex,label=right:$v^+$] (cross-v-plus)  at (8.3,0.45) {};
        \node[lift vertex,label=right:$v^-$] (cross-v-minus) at (8.3,-0.45) {};
        \draw[lift edge] (cross-u-plus) -- (cross-v-minus);
        \draw[lift edge] (cross-u-minus) -- (cross-v-plus);
        % \node[panel label] at (7.625,-0.9) {$s(uv)=-1$};
    \end{tikzpicture}
    \caption{The two possible $L_s$-liftings of a base edge $uv$: the parallel matching for $s(uv)=+1$ and the crossed matching for $s(uv)=-1$.}
    \label{fig:two-lift}
\end{figure}
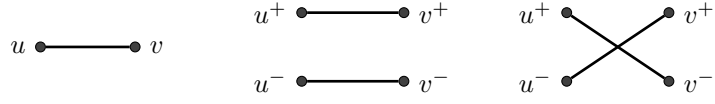

Let $G=(V,E)$ be a graph.
A \emph{signing} is a map $s\colon E\to\{+1,-1\}$.
Its \emph{signed adjacency matrix} $A_s$ is defined by
\[
(A_s)_{uv}\coloneqq \begin{cases}
    s(uv),&uv\in E,\\
    0,&uv\notin E.
\end{cases}
\]
The \emph{positive $2$-lift} $L_s(G)$ has vertex set $V\times\{+1,-1\}$ and edge set
\[
E(L_s(G))
=\bigl\{\{u^\sigma,v^{s(uv)\sigma}\}\colon
uv\in E,\ \sigma\in\{+1,-1\}\bigr\}.
\]
Here and throughout, $u^\sigma$ denotes the lifted copy $(u,\sigma)$ of a vertex $u$; also, we abbreviate $u^{+1}$ and $u^{-1}$ to $u^+$ and $u^-$ respectively.
The two vertices $u^+$ and $u^-$ form the \emph{lifted pair} corresponding to $u$.
The \emph{negative} (or \emph{complementary}) lift is $L_{-s}(G)$.
Thus, for each base edge $uv$, the matchings used by $L_s(G)$ and $L_{-s}(G)$ are disjoint and together form a copy of $K_{2,2}$ between the lifted pairs corresponding to $u$ and $v$ (see Figure~\ref{fig:two-lift}).

Let $A_+$ and $A_-$ be the adjacency matrices of the spanning subgraphs of $G$ formed by the positively and negatively signed edges, respectively.
Thus
\[
A_{+} = \frac{A(G)+A_{s}}{2},
\quad A_{-} = \frac{A(G)-A_{s}}{2}, 
\quad\text{and}\quad
A(L_s(G))=
\begin{pmatrix}
    A_+&A_-\\
    A_-&A_+
\end{pmatrix}.
\]

Note that the subspaces
\[
W_+\coloneqq\{(\mathbf{v},\mathbf{v})\colon\mathbf{v}\in\R^V\}
\quad\text{and}\quad
W_-\coloneqq\{(\mathbf{v},-\mathbf{v})\colon\mathbf{v}\in\R^V\}
\]
are invariant under $A(L_s(G))$.
Under the identifications $\mathbf{v}\mapsto(\mathbf{v},\mathbf{v})$ and $\mathbf{v}\mapsto(\mathbf{v},-\mathbf{v})$, the restrictions of $A(L_s(G))$ to $W_+$ and $W_-$ are represented by $A(G)$ and $A_s$, respectively.
Consequently, the following standard fact describes the spectrum of a $2$-lift.

\begin{fact}\label{fact:two-lift-spectrum}
    The spectrum of $L_s(G)$ is, with multiplicities, the union of the spectrum of $G$ and the spectrum of $A_s$.
    Similarly, the spectrum of $L_{-s}(G)$ is the union of the spectrum of $G$ and the spectrum of $A_{-s}=-A_s$.
\end{fact}

Our palette construction will use an arbitrarily long sequence of $2$-lifts, and at each step we choose a signing that controls the new eigenvalues of both complementary lifts.
The Bilu--Linial signing theorem~\cite[Theorem~3.1]{BiluLinial2006} provides the required signings.
In the form needed here, it states that there is an absolute constant $C$ such that, for every integer $d\ge2$, every graph of maximum degree at most $d$ admits a signing satisfying $\|A_s\|_{\op}\le C\sqrt{d\log^3 d}$. 
By the spectral decomposition above (Fact~\ref{fact:two-lift-spectrum}), this bound controls the new eigenvalues of $L_s(G)$ through $A_s$ and those of $L_{-s}(G)$ through $-A_s$.
Hence the same signing controls both complementary lifts, and iterating such signings yields the tower in the following proposition.

\begin{samepage}
    \begin{proposition}\label{prop:lift-tower}
        There exist an integer $d\ge160$ and a number $\rho\le d/100$ with the following property.
        For every integer $N_\star\ge1$ there is a sequence of connected $d$-regular graphs $D_0,D_1,D_2,\ldots$ and signings $s_i$ of $D_i$ such that for every integer $i\ge 0$ the following statements hold:
        \begin{enumerate}
            \item[(i)] $N_i\coloneqq |V(D_i)|=2^iN_0$ and $N_0\ge N_\star$;
            \item[(ii)] $D_{i+1}=L_{s_i}(D_i)$;
            \item[(iii)] if $D$ is $D_0$, $D_{i+1}$ or the complementary lift $D_{i+1}^-\coloneqq L_{-s_i}(D_i)$ then every eigenvalue of $D$ apart from the simple eigenvalue $d$ has absolute value at most $\rho$; in particular, the graph $D$ is connected.
        \end{enumerate}
    \end{proposition}
\end{samepage}

\begin{proof}
    Choose $d$ sufficiently large so that, in particular, $d\ge160$ and $\rho\coloneqq \max\left\{ 1,~C\sqrt{d\log^3d} \right\}\le\frac d{100}$. Take any $N_\star$.
    Start with the complete graph $K_{d+1}$ and repeatedly apply the Bilu--Linial signing theorem.
    Every $2$-lift of a simple $d$-regular graph is again simple and $d$-regular.
    The non-trivial eigenvalues of $K_{d+1}$ are all $-1$.
    At every positive lift the old eigenvalues are retained while the new eigenvalues have absolute value at most $\rho$.
    After sufficiently many preliminary lifts, the resulting graph has at least $N_\star$ vertices; call it $D_0$. Starting with $D_0$, continue the same procedure defining $D_{i+1}$ as the Bilu--Linial lift $L_{s_i}(D_i)$. As before, the claimed spectral properties hold.
    
    All obtained graphs (as well as their complementary lifts) are connected since the multiplicity of the maximum eigenvalue $d$ of a $d$-regular graph is the number of its connected components.
\end{proof}

Note that a $d$-regular graph $D=(V,E)$ has $d$ as a simple eigenvalue and all other eigenvalues at most $\rho$ in absolute value if and only if $\big\|A(D)-\frac d{|V|} J_V\big\|_{\op}\le\rho$.

%%%%%%%%%%%%%%%%%%%%%%%%%%%%%%%%%%%%%%%%%%%%%%%%%%%%
\section{Projective limits of finite palettes}
\label{sec:projective-limits}
%%%%%%%%%%%%%%%%%%%%%%%%%%%%%%%%%%%%%%%%%%%%%%%%%%%%
Recall that \cite{LiuPikhurkoIntervals} shows that the uniform Tur\'an densities arising from finite families of $3$-graphs are dense in a non-degenerate interval.
This does not by itself yield an interval in $\PiU$: it is not known whether $\PiU$ is closed, so the limit of a convergent sequence of such densities need not \emph{a priori} belong to $\PiU$.

The following proposition bypasses this difficulty.
Given a compatible sequence of finite palettes whose Lagrangians decrease to $x$, it constructs a single, possibly infinite, forbidden family $\cF$ with $\pivvv(\cF)=x$.
Thus, instead of proving that $\PiU$ is closed, it is enough for our purposes to construct suitable projective chains of finite palettes and compute their limiting Lagrangians.

% The proof treats the two inequalities separately.
% For the lower bound, the weighted palette construction supplies dense $P_m$-colorable $3$-graphs whose densities approach $x$.
% For the upper bound, we list all finite palettes $Q$ with $\lambda(Q)>x$ and use palette separation to place in $\cF$ a finite obstruction that is $Q$-colorable but not $P_m$-colorable for a suitable $m$.
% The arbitrary-family palette characterization then rules out every uniform density strictly larger than $x$.

\begin{proposition}
    \label{prop:projective-chain}
    Let $P_1,P_2,\ldots$ be finite palettes such that $P_{m+1}\to P_m$ for every $m\ge1$.
    If $\lambda(P_m)\downarrow x$, then $x\in\PiU$.
\end{proposition}

\begin{proof}
    If $x \in\{0,1\}$, the claim is immediate by taking $\cF$ to consist of the one-edge $3$-graph or to be the empty family.
    We may therefore assume $0<x<1$.
    For each $m$, set $\varepsilon_m\coloneqq \min\left\{\frac1m,\frac{\lambda(P_m)}2\right\}$. 
    Then $\varepsilon_m>0$ and $\varepsilon_m\to0$.
    
    There are only countably many finite palettes up to bijective relabeling of their color sets.
    Choose one representative of every isomorphism class with Lagrangian larger than $x$, and enumerate them as $Q_1,Q_2,\ldots$\,.
    For each $j$, choose $k_j$ such that $\lambda(P_{k_j})<\lambda(Q_j)$.
    By Fact~\ref{fact:palette-monotonicity}, there is no homomorphism from $Q_j$ to $P_{k_j}$ or to $\rev(P_{k_j})$.
    Hence Theorem~\ref{thm:palette-separation} gives a finite $3$-graph $F_j$ that is $Q_j$-colorable but not $P_{k_j}$-colorable.
    
    Set $g_0=0$.
    Inductively on $m=1,2,\ldots$\,, use Lemma~\ref{lem:weighted-palette}, with $n_0=g_{m-1}+1$, to choose a $P_m$-colorable $3$-graph $G_m$ such that $g_m\coloneqq |V(G_m)|>g_{m-1}$ and $G_m$ is uniformly $(\lambda(P_m)-\varepsilon_m,1/m)$-dense.
    
    For each integer $j\ge 1$, let $t_j$ be a positive integer such that $t_j|V(F_j)|>g_{k_j-1}$, and let $\widehat F_j$ be the \emph{balanced $t_j$-blow-up} of $F_j$.
    That is, replace every vertex $v$ of $F_j$ by a cluster $V_v$ of size $t_j$, and include exactly those triples that meet three distinct clusters whose corresponding vertices form an edge of $F_j$.
    Thus $|V(\widehat F_j)|>g_{k_j-1}$.
    
    The $3$-graph $\widehat F_j$ is $Q_j$-colorable: order the clusters according to a witnessing order of $F_j$, order the vertices arbitrarily within each cluster, and copy the witnessing pair color between every two distinct clusters.
    (Since no edge of $\widehat F_j$ contains two vertices from one cluster, we can color pairs inside clusters arbitrarily.)
    
    Furthermore, if $\widehat F_j$ were $P_{k_j}$-colorable, choose one vertex from every cluster.
    The induced subgraph is a copy of $F_j$, and restricting the witnessing order and pair coloring would make this copy $P_{k_j}$-colorable.
    Relabeling the copy by its cluster indices would then make $F_j$ itself $P_{k_j}$-colorable, a contradiction.
    
    Set $\cF\coloneqq \{\widehat F_j\colon j\ge1\}$. We claim that $\pivvv(\cF)=x$.

    First, let us show that every $3$-graph $G_m$ is $\cF$-free.
    Take any integer $j\ge 1$.
    If $k_j>m$, then the monotonicity of the sequence $(g_m)$ gives $|V(\widehat F_j)|>g_{k_j-1}\ge g_m=|V(G_m)|$, so $\widehat F_j$ cannot embed into $G_m$.
    If $k_j\le m$, composing the homomorphisms $P_m\to P_{m-1}\to\cdots\to P_{k_j}$ gives a homomorphism $P_m\to P_{k_j}$.
    By Fact~\ref{fact:palette-monotonicity}, every subgraph of the $P_m$-colorable graph $G_m$ is then $P_{k_j}$-colorable.
    Since $\widehat F_j$ is not $P_{k_j}$-colorable, it again cannot embed in $G_m$.
    
    We now verify the lower bound $\pivvv(\cF)\ge x$ with the quantifiers as in Definition~\ref{def:uniform-density}.
    Fix $0\le d<x$, $\eta>0$, and $n_0\in\N$.
    For all sufficiently large $m$, we have $g_m\ge n_0$, $\frac1m\le\eta$, and $\lambda(P_m)-\varepsilon_m\ge d$.
    Choose such an $m$.
    By the claim proved above, $G_m$ is $\cF$-free.
    Moreover, its constructed density parameter is at least $d$ and its error parameter is at most $\eta$, so it is uniformly $(d,\eta)$-dense.
    Hence $d\in A(\cF)$.
    Since this holds for every $d\in [0,x)$, we obtain $\pivvv(\cF)\ge x$.
    
    Suppose on the contrary that $\pivvv(\cF)>x$.
    Choose a real number $t$ with $x<t<\pivvv(\cF)$.
    By Theorem~\ref{thm:palette-characterization}, the supremum in that theorem is larger than $t$, so there exists a finite palette $Q$ such that $d(Q)>t>x$ and no member of $\cF$ is $Q$-colorable.
    Since $d(Q)\le\lambda(Q)$, the palette $Q$ is isomorphic to some $Q_j$.
    Transporting the color labels through this isomorphism, the $Q_j$-coloring of $\widehat F_j\in\cF$ is also a $Q$-coloring, a contradiction.
    
    Therefore $\pivvv(\cF)=x$, and hence $x\in\PiU$.
\end{proof}

%%%%%%%%%%%%%%%%%%%%%%%%%%%%%%%%%%%%%%%%%%%%%%%%%%%%
\section{The palette lift and its losses}
\label{sec:palette-lift}
%%%%%%%%%%%%%%%%%%%%%%%%%%%%%%%%%%%%%%%%%%%%%%%%%%%%

This section uses the graph tower from Proposition~\ref{prop:lift-tower} to build a recursive tower of palettes.
At every stage, the uniform weighting must remain the unique maximizer, while the construction must also permit small losses that can be chosen independently at successive levels.
Here is an informal overview of the construction; all objects used in it are defined precisely in the subsections below.
We start with a palette built from the initial regular graph in the tower (Section~\ref{sec:root-palette}).
At each subsequent level, every color is replaced by two lifted copies.
We then delete certain repeated-color patterns prescribed by the complementary graph lift, while the other graph lift is retained to govern the next recursive step (Section~\ref{sec:spectral-palette-lift} and Lemma~\ref{lem:exact-lift}).
Finally, we may delete up to four suitable tripartite perfect matchings (Section~\ref{sec:matching-deletions}).
The complementary-lift deletion is forced by the graph tower, whereas the matching deletions will later be chosen independently to create the different branches of the palette construction (Section~\ref{sec:digit-matchings}).
The natural projection from the doubled color set to its parent will give a palette homomorphism at every level (Section~\ref{sec:tower-stability}).

We first define the root palette and prove that its uniform weighting is quantitatively stable.
We next give the formal definition of the spectral palette lift and establish its exact loss identity.
After introducing the optional matching deletions, we use spectral expansion to control the loss caused by an arbitrary imbalance between lifted copies.
The final subsection assembles these one-step estimates into a precise recursive definition of a palette tower and proves uniform stability simultaneously at every level.

\subsection{The root palette}
\label{sec:root-palette}
An ordered triple of palette colors is called a \emph{repeated-color pattern} if it contains exactly two colors, one appearing twice and the other once.
Thus a pair $uv$ of colors can produce the following six repeated-color patterns:
\begin{equation}\label{eq:repeated-color-patterns}
    (u,u,v),\ (u,v,u),\ (v,u,u),\
    (v,v,u),\ (v,u,v),\ (u,v,v).
\end{equation}
Monochromatic triples $(u,u,u)$ are treated separately and are not included in this terminology.

Let $D=(V,E)$ be a simple $d$-regular graph on a color set $V$ of size $N$.
Define the \emph{root palette} $P(D)$ on the color set $V$ as follows:
\begin{itemize}
    \item every ordered triple of three pairwise distinct colors is admissible;
    \item for every edge $uv\in E(D)$, the six repeated-color patterns in Equation~\eqref{eq:repeated-color-patterns} are admissible;
    \item no monochromatic triple $(u,u,u)$ is admissible.
\end{itemize}
For a weighting $\mathbf{x}\in\Delta_V$, the palette polynomial of $P(D)$ is
\begin{equation}\label{eq:root-polynomial}
    \Lambda_{P(D)}(\mathbf{x})=1-3\sum_{v\in V} x_v^2+2\sum_{v\in V}x_v^3
    +3\sum_{uv\in E}(x_u^2x_v+x_v^2x_u).
\end{equation}
Here the first three terms give the total contribution of the patterns with three pairwise distinct colors, while the final sum accounts for the repeated-color patterns associated with the edges of $D$.
Since $D$ is $d$-regular and hence has $dN/2$ edges, evaluating Equation~\eqref{eq:root-polynomial} at the uniform weighting gives
\begin{equation}\label{eq:root-uniform}
    \Lambda_{P(D)}(\mathbf{u}_V)=1-\frac3N+\frac{2+3d}{N^2}.
\end{equation}

The next lemma shows that this value is quantitatively stable: moving away from the uniform weighting decreases the palette polynomial by at least a quadratic amount.

\begin{lemma}\label{lem:root-stability}
    For every fixed $d$ there are constants $N_{\mathrm{root}}(d)$ and $\kappa_d>0$ such that, whenever $D=(V,E)$ is a $d$-regular graph on $N\ge N_{\mathrm{root}}(d)$ vertices,
    \[
    \Lambda_{P(D)}(\mathbf{u}_V)-\Lambda_{P(D)}(\mathbf{x})\ge\kappa_d\|\mathbf{x}-\mathbf{u}_V\|_2^2
    \]
    for every $\mathbf{x}\in\Delta_V$.
\end{lemma}

\begin{proof} Set $\theta\coloneqq \frac1{2(2+3d)}$, $\kappa_d\coloneqq \frac{\theta^3}{2}$, and $N_{\mathrm{root}}(d)\coloneqq 6\theta^{-3}$.
    Let $\varphi(t) \coloneqq -3t^2+(2+3d)t^3$.
    For every $t\in [0,\theta]$, we have that $\varphi''(t)=-6+6(2+3d)t\le-3$. 
    
    Let $N$ and $D$ be as in the lemma. It holds that $1/N\le1/N_{\mathrm{root}}(d)\le \theta$.	
    For $a,b\ge0$, we have $a^2b+b^2a\le a^3+b^3$, because the difference is $(a+b)(a-b)^2$.
    Since $D$ is $d$-regular, Equation~\eqref{eq:root-polynomial} gives
    \begin{equation}\label{eq:root-upper}
        \Lambda_{P(D)}(\mathbf{x})\le1-3\sum_{v\in V} x_v^2+(2+3d)\sum_{v\in V}x_v^3, 
    \end{equation}
    and equality holds at $\mathbf{u}_V$.

    Suppose first that $\max_{v\in V} x_v\le\theta$. 
    For each $v\in V$, Taylor's theorem for $1/N$ and $x_v$, both in $[0,\theta]$, gives
    \[
    \varphi(x_v)\le \varphi\left(\frac{1}{N}\right)
    +\varphi'\left(\frac{1}{N}\right)
    \left(x_v-\frac{1}{N}\right)
    -\frac{3}{2}\left(x_v-\frac{1}{N}\right)^2.
    \]
    Summing these inequalities over all $v\in V$ and using $\sum_{v\in V}(x_v-1/N)=0$ yields
    \[
    \sum_{v\in V} \varphi(x_v)
    \le N\varphi\left(\frac{1}{N}\right)
    -\frac{3}{2}\|\mathbf{x}-\mathbf{u}_V\|_2^2.
    \]
    Together with Equation~\eqref{eq:root-upper}, and using equality in that equation at $\mathbf{u}_V$, this gives
    \[
    \Lambda_{P(D)}(\mathbf{u}_V)-\Lambda_{P(D)}(\mathbf{x})
    \ge\frac{3}{2}\|\mathbf{x}-\mathbf{u}_V\|_2^2.
    \]
    Since $0<\theta<1$ and hence $\kappa_d=\theta^3/2<3/2$, this is in fact stronger than the claimed bound.
    
    Finally, suppose that $\max_{v\in V} x_v>\theta$. Note that no admissible ordered triple is monochromatic.
    Since the sum of the monomials corresponding to all ordered triples is $1$, we have $\Lambda_{P(D)}(\mathbf{x})\le1-\sum_{v\in V}x_v^3\le1-\theta^3$. 
    By Equation~\eqref{eq:root-uniform} and the choice of $N_{\mathrm{root}}(d)$, we obtain $\Lambda_{P(D)}(\mathbf{u}_V)\ge1-\frac3N\ge1-\frac{\theta^3}{2}$. 
    Since $\|\mathbf{x}-\mathbf{u}_V\|_2^2\le1$, the asserted inequality follows with coefficient $\kappa_d=\theta^3/2$.
\end{proof}

\subsection{The spectral palette lift}
\label{sec:spectral-palette-lift}

We next describe the recursive step that passes from a parent palette to a palette on the doubled color set.
The complementary lift specifies which repeated-color patterns are deleted, and the exact lift identity below expresses their total contribution as a cubic graph functional.

Suppose that a palette $P=(V(D),T)$ contains all permutations of $(u,u,v)$ and $(v,v,u)$ for every $uv\in E(D)$.
We then call $D$ a \emph{distinguished repeated-color graph} of $P$; this designation does not exclude the presence of other repeated-color patterns in $P$.
Let $s$ be a signing of $D$, and let $\operatorname{pr}\colon V(D)\times\{+1,-1\}\to V(D)$ be the natural projection.
First double every color of $P$ without deleting any patterns; the resulting admissible set is
\[
T^{\mathrm{dbl}}
\coloneqq \left\{ (\widetilde a,\widetilde b,\widetilde c)\colon
(\operatorname{pr}(\widetilde a),\operatorname{pr}(\widetilde b),
\operatorname{pr}(\widetilde c))\in T \right\}.
\]
Thus a child triple is admissible precisely when its projection is admissible in $P$.
Write $P^{\mathrm{dbl}}\coloneqq (V(D)\times\{+1,-1\},T^{\mathrm{dbl}})$.
If $\mathbf{x}$ is a weighting of the doubled color set and $\mathbf{y}\coloneqq\operatorname{pr}_*\mathbf{x}$, then $\Lambda_{P^{\mathrm{dbl}}}(\mathbf{x})=\Lambda_P(\mathbf{y})$.
Conversely, given any $\mathbf{y}\in\Delta_{V(D)}$, the weighting defined by $x_{u^+}=y_u$ and $x_{u^-}=0$ satisfies $\operatorname{pr}_*\mathbf{x}=\mathbf{y}$.
It follows in both directions that color doubling preserves the Lagrangian:
\[
\lambda(P^{\mathrm{dbl}})=\lambda(P).
\]
For every edge $uv\in E(D)$ and every sign $\sigma$, delete all permutations of $(u^\sigma,u^\sigma,v^{-s(uv)\sigma})$ and of $(v^{-s(uv)\sigma},v^{-s(uv)\sigma},u^\sigma)$. 
Denote the resulting \emph{spectral palette lift} by $L_s(P,D)$. Thus $L_s(P,D)$ is obtained from $P^{\mathrm{dbl}}$ by deleting all repeated-color triples coming from the edges of the complementary lift $L_{-s}(D)$, exactly 12 triples for each edge of $D$. 
The projection $\operatorname{pr}$ is a palette homomorphism $L_s(P,D)\to P$.
Among the lifted copies of the repeated-color patterns distinguished by $D$, the retained patterns are exactly those associated with $L_s(D)$.
Thus the child palette $L_s(P,D)$ carries $L_s(D)$ as a distinguished repeated-color graph.

For a graph $G$ and a nonnegative vector $\mathbf{x}$ indexed by $V(G)$, let $\mathbf{x}^{\square}\coloneqq (x_v^2)_{v\in V(G)}$ denote the coordinatewise square of $\mathbf{x}$, and define the \emph{cubic edge functional}
\begin{equation}\label{eq:cubic-functional}
    \mathcal{E}_G(\mathbf{x})\coloneqq \sum_{uv\in E(G)}x_ux_v(x_u+x_v)
    =\langle \mathbf{x}^{\square},A(G)\mathbf{x}\rangle.
\end{equation}

The following identity separates the unchanged contribution inherited from the parent palette from the exact loss caused by the patterns deleted along the complementary lift.

\begin{lemma}\label{lem:exact-lift}
    Let $D$ be a distinguished repeated-color graph of a palette $P$, let $s$ be a signing of $D$, and put $P^+\coloneqq L_s(P,D)$ and $D^-\coloneqq L_{-s}(D)$.
    For a weighting $\mathbf{x}$ of the child colors, let $\mathbf{y}\coloneqq \operatorname{pr}_*\mathbf{x}$ be its push-forward to $V(D)$; explicitly, $y_u=x_{u^+}+x_{u^-}$.
    Then
    \begin{equation}\label{eq:lift-identity}
        \Lambda_{P^+}(\mathbf{x})=\Lambda_P(\mathbf{y})-3\mathcal{E}_{D^-}(\mathbf{x}).
    \end{equation}
    If, in addition, $D$ is $d$-regular on $N$ vertices, then at the uniform weighting $\mathbf{u}_{2N}$ the loss is
    \begin{equation}\label{eq:lift-uniform-loss}
        3\mathcal{E}_{D^-}(\mathbf{u}_{2N})=\frac{3d}{4N^2}.
    \end{equation}
\end{lemma}

\begin{proof}
    The eight lifted copies of each parent triple contribute the product of its three aggregate weights, so before the deletions the palette polynomial of the doubled palette evaluates to $\Lambda_P(\mathbf{y})$ at $\mathbf{x}$.
    The deleted repeated-color triples are exactly the six ordered patterns associated with every edge of the negative lift $D^-$.
    Their total contribution to the palette polynomial is $3\mathcal{E}_{D^-}(\mathbf{x})$, proving Equation~\eqref{eq:lift-identity}.
    Under the additional assumption that $D$ is $d$-regular on $N$ vertices, the graph $D^-$ is $d$-regular on $2N$ vertices, so $\mathcal{E}_{D^-}(\mathbf{u}_{2N})=d/(2N)^2$.
\end{proof}

\subsection{Tripartite perfect matchings}
\label{sec:matching-deletions}

The complementary-lift deletion is determined by the graph tower and therefore creates no choice between branches.
To construct an interval, we also need optional deletions that can be chosen independently at different levels.
Tripartite perfect matchings provide these branching choices: their cost is easy to compute at the uniform weighting and remains controlled for arbitrary weightings.

Let $C$ be a color set of size $N$.
A \emph{tripartite perfect matching} on $C$ is a collection
\[
\cM=\{(a_i,b_i,c_i)\colon i\in[N]\}\subseteq C^3
\]
for which each of the maps $i\mapsto a_i$, $i\mapsto b_i$, and $i\mapsto c_i$ is a bijection from $[N]$ to $C$.
Equivalently, if the three coordinate positions are viewed as three labeled copies of $C$, then the triples in $\cM$ form a perfect matching in the complete tripartite $3$-graph.
Two tripartite perfect matchings on $C$ are \emph{pattern-disjoint} if they are disjoint as subsets of $C^3$.
A collection of tripartite perfect matchings is \emph{pairwise pattern-disjoint} if every two distinct members are pattern-disjoint.
If $P=(C,T)$ is a palette, then a tripartite perfect matching $\cM$ on $C$ is \emph{admissible in $P$} if $\cM\subseteq T$.
For a weighting $\mathbf{x}\in\Delta_C$, define
\begin{equation}\label{eq:R}	
    R_{\cM}(\mathbf{x})\coloneqq \sum_{i=1}^N x_{a_i}x_{b_i}x_{c_i}.
\end{equation}
If the patterns in $\cM$ are deleted from a palette, then its palette polynomial decreases by exactly $R_{\cM}(\mathbf{x})$.
At the uniform weighting, each of the $N$ patterns has weight $N^{-3}$, and hence $R_{\cM}(\mathbf{u}_N)=1/N^2$. 

The next lemma shows that, away from the uniform weighting, the matching loss can fall below this value only by an amount controlled quadratically by the deviation from uniformity.

\begin{lemma}\label{lem:matching-bound}
    For every tripartite perfect matching $\cM$ on $N$ colors and every $\mathbf{x}\in\Delta_C$, 
    \begin{equation}\label{eq:matching-bound}
        R_{\cM}(\mathbf{x})-\frac1{N^2}
        \ge-\frac3N\|\mathbf{x}-\mathbf{u}_N\|_2^2.
    \end{equation}
\end{lemma}

\begin{proof}
    We first prove the auxiliary inequality
    \begin{equation}\label{eq:abc}
        abc\ge a+b+c-2-\bigl((a-1)^2+(b-1)^2+(c-1)^2\bigr)
    \end{equation}
    for all $a,b,c\ge0$.
    The difference between the two sides is
    \[
    F(a,b,c)\coloneqq a^2+b^2+c^2+abc-3a-3b-3c+5.
    \]
    The quadratic terms imply $F(a,b,c)\to\infty$ whenever $a^2+b^2+c^2\to\infty$ in $[0,\infty)^3$ (the term $abc$ is nonnegative).
    Thus $F$ is coercive and attains a minimum on this closed orthant.
    On the boundary, for example when $c=0$,
    \[
    F(a,b,0)=\left(a-\frac32\right)^2
    +\left(b-\frac32\right)^2+\frac12>0,
    \]
    and the other boundary cases are symmetric.
    At an interior critical point, $2a+bc=2b+ac=2c+ab=3$. 
    Subtracting pairs gives $(a-b)(2-c)=(b-c)(2-a)=(c-a)(2-b)=0$. 
    If none of $a,b,c$ equals $2$, then $a=b=c$, and the critical-point equation becomes $2a+a^2=3$.
    Its roots are $1$ and $-3$, so nonnegativity gives $a=b=c=1$.
    If, say, $c=2$, then the third critical-point equation $2c+ab=3$ is impossible for $a,b\ge0$.
    The other cases are symmetric.
    Thus the only interior critical point is $(1,1,1)$, where $F=0$, and Equation~\eqref{eq:abc} follows.
    
    Put $w_c=Nx_c$ and $H=\sum_c(w_c-1)^2$.
    Apply Equation~\eqref{eq:abc} to every triple $(w_{a_i},w_{b_i},w_{c_i})$ and sum over $i$.
    Since every color occurs exactly once in each coordinate, we have $\sum_iw_{a_i}w_{b_i}w_{c_i}\ge N-3H$. 
    Dividing by $N^3$ and using $H=N^2\|\mathbf{x}-\mathbf{u}_N\|_2^2$ proves the result.
\end{proof}

%%%%%%%%%%%%%%%%%%%%%%%%%%%%%%%%%%%%%%%%%%%%%%%%%%%%
\subsection{A global cubic expander inequality}
\label{sec:cubic-expander}
%%%%%%%%%%%%%%%%%%%%%%%%%%%%%%%%%%%%%%%%%%%%%%%%%%%%

By Lemma~\ref{lem:exact-lift}, the loss caused by the complementary-lift deletion is exactly the cubic functional $3\mathcal{E}_G(\mathbf{x})$.
For a $d$-regular graph $G$ on $M$ vertices, this loss equals $3d/M^2$ at the uniform weighting.
In the following lemma, in order to compare the loss under an arbitrary weighting with its value at the uniform weighting, we separate the behavior within and between the prescribed pairs: the vector $\mathbf{z}$ measures the imbalance between the two entries of each pair, while $\mathbf{y}$ records their total weight.
The spectral gap of $G$ forces a positive loss proportional to $\|\mathbf{z}\|_2^2$, up to an error controlled by $\|\mathbf{y}-\mathbf{u}_N\|_2^2$.
The next lemma gives this estimate globally, for every weighting, and is the analytic core of the construction.

\begin{lemma}\label{lem:cubic-expander}
    Let $G$ be a $d$-regular graph on $M=2N$ vertices, paired as $v^+,v^-$ for $v\in[N]$.
    Suppose that $\left\|A(G)-\frac dM J\right\|_{\op}\le\rho\le\frac d{100}$. 
    For $\mathbf{x}\in\Delta_{V(G)}$, define vectors $\mathbf{y},\mathbf{z}\in\R^N$ by $y_v \coloneqq x_{v^+}+x_{v^-}$ and $z_v \coloneqq x_{v^+}-x_{v^-}$.
    Then
    \begin{equation}\label{eq:cubic-expander}
        3\mathcal{E}_G(\mathbf{x})-\frac{3d}{M^2}
        \ge\frac12\frac dN\|\mathbf{z}\|_2^2
        -\frac{27}{20}\frac dN\|\mathbf{y}-\mathbf{u}_N\|_2^2.
    \end{equation}
\end{lemma}

\begin{proof}
    Define vectors $\mathbf{q}\in\R^{V(G)}$ and $\mathbf{w},\mathbf{b}\in\R^N$ by $q_{v^\pm}\coloneqq Mx_{v^\pm}$, $w_v\coloneqq Ny_v$, and $b_v\coloneqq Nz_v$.
    Then $q_{v^\pm}=w_v\pm b_v$, $\sum_vw_v=N$, and $|b_v|\le w_v$.
    Write $H\coloneqq \|\mathbf{w}-\1\|_2^2$ and $B\coloneqq \|\mathbf{b}\|_2^2$.
    Since $\mathcal{E}_G$ is homogeneous of degree three, we have $\mathcal{E}_G(\mathbf{x})=M^{-3}\mathcal{E}_G(\mathbf{q})$, while $\|\mathbf{z}\|_2^2=N^{-2}B$ and $\|\mathbf{y}-\mathbf{u}_N\|_2^2=N^{-2}H$.
    Using $M=2N$, multiplying Equation~\eqref{eq:cubic-expander} by $M^3/3$ shows that it is equivalent to
    \begin{equation}\label{eq:cubic-normalized}
        \mathcal{E}_G(\mathbf{q})-dM\ge\frac43dB-\frac{18}{5}dH.
    \end{equation}
    
    We first establish a coarse bound.
    Put $\mathbf{q}=\1+\mathbf{g}$, so $\mathbf{g}\perp\1$ and $g_i\ge-1$.
    The two coordinates belonging to each pair give
    \[
    \|\mathbf{g}\|_2^2
    =\sum_v\bigl((w_v+b_v-1)^2+(w_v-b_v-1)^2\bigr)
    =2(H+B).
    \]
    Expanding Equation~\eqref{eq:cubic-functional}, we obtain 
    \begin{equation}\label{eq:cubic-expand-g}
        \mathcal{E}_G(\mathbf{q})-dM
        =d\|\mathbf{g}\|_2^2+2\langle \mathbf{g},A(G)\mathbf{g}\rangle
        +\langle \mathbf{g}^{\square},A(G)\mathbf{g}\rangle.
    \end{equation}
    Every coordinate of $A(G)\mathbf{g}=A(G)\mathbf{q}-d\1$ is at least $-d$, and therefore the last term is at least $-d\|\mathbf{g}\|_2^2$.
    Since $\mathbf{g}\perp\1$, the spectral assumption gives
    \begin{equation}\label{eq:coarse-bound}
        \mathcal{E}_G(\mathbf{q})-dM\ge-2\rho\|\mathbf{g}\|_2^2=-4\rho(H+B).
    \end{equation}
    
    Suppose first that $H>\frac25B$.
    Then $B<\frac52H$, and Equation~\eqref{eq:coarse-bound} gives $\mathcal{E}_G(\mathbf{q})-dM\ge-\frac{14}{100}dH$.
    On the other hand, $\frac43B-\frac{18}{5}H<-\frac4{15}H$. 
    Thus Equation~\eqref{eq:cubic-normalized} holds in this case.
    
    It remains to consider $H\le\frac25B$.
    Since $|b_v|\le w_v$, $B\le\sum_vw_v^2=N+H$, and consequently $H\le\frac23N$.
    Call the pair $\{v^+,v^-\}$ \emph{heavy} if $w_v>10$, and let $S\coloneqq \{v\colon w_v>10\}$ be the set of indices of the heavy pairs and $W\coloneqq \sum_{v\in S}w_v$ their total weight.
    For $w\ge10$, we have $w\le\frac{10}{81}(w-1)^2$ and $w^2\le\frac{100}{81}(w-1)^2$. 
    Hence
    \begin{equation}\label{eq:heavy-bounds}
        W\le\frac{10}{81}H
        \quad\text{and}\quad
        \sum_{v\in S}b_v^2\le\frac{100}{81}H.
    \end{equation}
    
    Let $\mathbf{q}'$ be obtained from $\mathbf{q}$ by setting both coordinates of each heavy pair to zero.
    Since the coefficients of $\mathcal{E}_G$ are nonnegative, $\mathcal{E}_G(\mathbf{q})\ge \mathcal{E}_G(\mathbf{q}')$.
    The mean of $\mathbf{q}'$ is $\mu=1-\frac WN \ge1-\frac{10}{81}\cdot\frac23 =\frac{223}{243}$. 
    Write $\mathbf{q}'=\mu\1+\mathbf{h}$ with $\mathbf{h}\perp\1$.
    On a non-heavy pair we have $0\le q_{v^\pm}=w_v\pm b_v\le2w_v\le20$, while on a heavy pair both coordinates of $\mathbf{q}'$ are zero.
    Since $0\le\mu\le1$, it follows that $\|\mathbf{h}\|_\infty\le20$.
    Expanding as before,
    \begin{equation}\label{eq:cubic-expand-h}
        \mathcal{E}_G(\mathbf{q}')=dM\mu^3
        +\mu\bigl(d\|\mathbf{h}\|_2^2+2\langle \mathbf{h},A(G)\mathbf{h}\rangle\bigr)
        +\langle \mathbf{h}^{\square},A(G)\mathbf{h}\rangle.
    \end{equation}
    Since $\mathbf{h}\perp\1$, we have $J\mathbf{h}=0$ and hence
    \[
    \|A(G)\mathbf{h}\|_2
    =\left\|\left(A(G)-\frac dM J\right)\mathbf{h}\right\|_2
    \le\rho\|\mathbf{h}\|_2.
    \]
    Consequently, $|\langle \mathbf{h}^{\square},A(G)\mathbf{h}\rangle| \le\|\mathbf{h}^{\square}\|_2\|A(G)\mathbf{h}\|_2 \le20\rho\|\mathbf{h}\|_2^2$. 
    Combining this with $\langle\mathbf{h},A(G)\mathbf{h}\rangle\ge-\rho\|\mathbf{h}\|_2^2$, all terms after $dM\mu^3$ in Equation~\eqref{eq:cubic-expand-h} are at least
    \[
    \bigl(\mu(d-2\rho)-20\rho\bigr)\|\mathbf{h}\|_2^2
    \ge d\left(\frac{223}{243}\cdot\frac{49}{50}-\frac15\right)\|\mathbf{h}\|_2^2
    =\frac{8497}{12150}d\|\mathbf{h}\|_2^2
    >\frac23d\|\mathbf{h}\|_2^2.
    \]
    Moreover, $1-\mu^3\le3(1-\mu)=3W/N$, so $dM\mu^3-dM\ge-6dW\ge-\frac{60}{81}dH$. 
    
    Let $B_{\mathrm{good}}\coloneqq \sum_{v\notin S}b_v^2$.
    For every non-heavy pair, the two coordinates of $\mathbf{h}$ contribute at least $2b_v^2$, and therefore $\|\mathbf{h}\|_2^2 \ge2B_{\mathrm{good}} \ge2\left(B-\frac{100}{81}H\right)$. 
    Combining these estimates gives
    \[
    \mathcal{E}_G(\mathbf{q})-dM
    \ge\frac43dB
    -\left(\frac{60}{81}+\frac43\frac{100}{81}\right)dH
    =\frac43dB-\frac{580}{243}dH.
    \]
    Since $580/243<18/5$, this is at least $\frac43dB-\frac{18}{5}dH$.
    This proves Equation~\eqref{eq:cubic-normalized}, and rescaling proves Equation~\eqref{eq:cubic-expander}.
\end{proof}

%%%%%%%%%%%%%%%%%%%%%%%%%%%%%%%%%%%%%%%%%%%%%%%%%%%%
\subsection{Multiscale stability of the palette tower}
\label{sec:tower-stability}
%%%%%%%%%%%%%%%%%%%%%%%%%%%%%%%%%%%%%%%%%%%%%%%%%%%%

We now combine the preceding one-step estimates to prove stability for palettes obtained after arbitrarily many lifts.
At each step, the cubic expander inequality gives a positive penalty for imbalance within the newly formed lifted pairs, together with an error controlled by the deviation of the pushed-forward weighting from uniformity.
The matching bound contributes additional errors of the same quadratic type.
Since the number of colors doubles at every step, these error terms occur on geometrically decreasing scales.
After summing over the tower, they can be absorbed by the stability of the root palette and the positive lifted-pair penalties.

Fix the graph tower from Proposition~\ref{prop:lift-tower}.
Let $P_0=P(D_0)$.
By definition, $D_0$ is a distinguished repeated-color graph of $P_0$.
Inductively, suppose $P_i$ has color set $V(D_i)$ and $D_i$ is a distinguished repeated-color graph of $P_i$.
Form the spectral lift $\widetilde P_{i+1}\coloneqq L_{s_i}(P_i,D_i)$.
Delete from $\widetilde P_{i+1}$ the union of at most four pairwise pattern-disjoint admissible tripartite perfect matchings, each consisting entirely of triples of three pairwise distinct colors, and call the resulting palette $P_{i+1}$.
These deletions do not remove any distinguished repeated-color pattern, so the induction may continue.
Write $\operatorname{pr}_{i+1,i}\colon V(D_{i+1})\to V(D_i)$ for the natural projection; it is a homomorphism $P_{i+1}\to P_i$.
We call any sequence obtained by this recursion a \emph{palette tower}.

For $0\le i\le n$, let $\operatorname{pr}_{n,i}\colon V(D_n)\to V(D_i)$ be the composition of the appropriate $n-i$ projections, with $\operatorname{pr}_{n,n}$ being the identity map.
For a weighting $\mathbf{x}$ on $V(D_n)$, define $\mathbf{x}^{(i)}\coloneqq(\operatorname{pr}_{n,i})_*\mathbf{x}$ for $0 \le i \le n$, and put $\mathbf{h}_i\coloneqq \mathbf{x}^{(i)}-\mathbf{u}_{N_i}$. (Recall that $N_i=|V(D_i)|=2^iN_0$.)
For each $0\le i<n$ and $v\in V(D_i)$, define the lifted-pair difference vector $\boldsymbol{\Delta}_i$ by $\Delta_i(v)\coloneqq x^{(i+1)}_{v^+}-x^{(i+1)}_{v^-}$.
Then
\begin{equation}\label{eq:orthogonal-split}
    \|\mathbf{h}_{i+1}\|_2^2
    =\frac12\|\mathbf{h}_i\|_2^2+\frac12\|\boldsymbol{\Delta}_i\|_2^2.
\end{equation}

The following theorem sums the one-step estimates over the entire tower and gives a uniform quadratic penalty for every departure from the uniform weighting, independently of the permitted matching deletions.

\begin{theorem}\label{thm:tower-stability}
    Let $d$ and $\rho$ be as in Proposition~\ref{prop:lift-tower}, let $\kappa_d$ be as in Lemma~\ref{lem:root-stability}, and put $B_d\coloneqq 27/20+12/d$.
    Suppose that
    \begin{equation}\label{eq:N0-condition}
        N_0\ge \max\left\{N_{\mathrm{root}}(d), ~\tfrac{8B_d\,d}{3\kappa_d} \right\}.
    \end{equation}
    Then every palette tower $(P_i)_{i=0}^\infty$ defined above satisfies, for every $n\ge0$ and every weighting $\mathbf{x}$ on $V(D_n)$,
    \begin{equation}\label{eq:tower-stability}
        \Lambda_{P_n}(\mathbf{u}_{N_n})-\Lambda_{P_n}(\mathbf{x})
        \ge\frac{\kappa_d}{2}\|\mathbf{h}_0\|_2^2
        +\frac1{40}\sum_{i=0}^{n-1}\frac d{N_i}\|\boldsymbol{\Delta}_i\|_2^2.
    \end{equation}
    In particular, $\mathbf{u}_{N_n}$ is the unique maximizer of $\Lambda_{P_n}$ on $\Delta_{V(D_n)}$, and hence $\lambda(P_n)=\Lambda_{P_n}(\mathbf{u}_{N_n})$.
\end{theorem}

\begin{proof}
    For $i\ge1$, let $\cL_i$ be the collection of tripartite perfect matchings deleted when $P_i$ is formed, and put $\cL_0=\varnothing$.
    Thus $|\cL_i|\le4$.
    Since the matchings in each $\cL_i$ are pairwise pattern-disjoint, their contributions to the palette polynomial add without multiplicity.
    More precisely, for $0\le i<n$,
    \[
    \Lambda_{P_{i+1}}(\mathbf{x}^{(i+1)})
    =\Lambda_{P_i}(\mathbf{x}^{(i)})-3\mathcal{E}_{D_{i+1}^-}(\mathbf{x}^{(i+1)})
    -\sum_{\cM\in\cL_{i+1}}R_{\cM}(\mathbf{x}^{(i+1)}).
    \]
    At $\mathbf{u}_{N_{i+1}}$, the corresponding losses are $3d/N_{i+1}^2$ and $N_{i+1}^{-2}$ per matching in $\cL_{i+1}$, respectively.
    Subtracting the two recursions and summing over $i$ gives the exact telescoping identity
    \begin{align}
        & \Lambda_{P_n}(\mathbf{u}_{N_n})-\Lambda_{P_n}(\mathbf{x}) \notag \\
        & = \Lambda_{P_0}(\mathbf{u}_{N_0})-\Lambda_{P_0}(\mathbf{x}^{(0)})
        +\sum_{i=0}^{n-1}
        \left(3\mathcal{E}_{D_{i+1}^-}(\mathbf{x}^{(i+1)})-\frac{3d}{N_{i+1}^2}\right)
        +\sum_{i=1}^{n}\sum_{\cM\in\cL_i}
        \left(R_{\cM}(\mathbf{x}^{(i)})-\frac1{N_i^2}\right).\label{eq:telescoping}
    \end{align}
    The root term $\Lambda_{P_0}(\mathbf{u}_{N_0})-\Lambda_{P_0}(\mathbf{x}^{(0)})$ is bounded below by $\kappa_d\|\mathbf{h}_0\|_2^2$ by Lemma~\ref{lem:root-stability}.
    For each $0\le i<n$, applying Lemma~\ref{lem:cubic-expander} to $D_{i+1}^-$, with the lifted pairs $\{v^+,v^-\}$ as the prescribed pairing, gives
    \[
    3\mathcal{E}_{D_{i+1}^-}(\mathbf{x}^{(i+1)})-\frac{3d}{N_{i+1}^2}
    \ge\frac12\frac d{N_i}\|\boldsymbol{\Delta}_i\|_2^2
    -\frac{27}{20}\frac d{N_i}\|\mathbf{h}_i\|_2^2.
    \]
    For each $1\le i\le n$, Lemma~\ref{lem:matching-bound} shows that the at most four matching deletions at level $i$ contribute at least $-\frac{12}{N_i}\|\mathbf{h}_i\|_2^2 = -\frac{12}{d}\frac d{N_i}\|\mathbf{h}_i\|_2^2$.
    Thus, for each interior index $1\le i\le n-1$, the error term $-\frac{27}{20}\frac d{N_i}\|\mathbf{h}_i\|_2^2$ from the cubic expander estimate and the error term $-\frac{12}{N_i}\|\mathbf{h}_i\|_2^2$ from the matching-deletion estimate have total contribution
    \[
    -\frac{27}{20}\frac d{N_i}\|\mathbf{h}_i\|_2^2
    -\frac{12}{N_i}\|\mathbf{h}_i\|_2^2
    =-B_d\frac d{N_i}\|\mathbf{h}_i\|_2^2,
    \quad\text{with}\quad B_d\coloneqq\frac{27}{20}+\frac{12}{d}.
    \]
    At $i=0$ there is no matching-deletion error, while at $i=n$ there is no cubic-expander error, so the same constant $B_d$ also bounds the negative error at each endpoint.
    We obtain
    \begin{equation}\label{eq:tower-pre-sum}
        \Lambda_{P_n}(\mathbf{u}_{N_n})-\Lambda_{P_n}(\mathbf{x})
        \ge\kappa_d\|\mathbf{h}_0\|_2^2
        +\frac12\sum_{i=0}^{n-1}\frac d{N_i}\|\boldsymbol{\Delta}_i\|_2^2
        -B_d\sum_{i=0}^{n}\frac d{N_i}\|\mathbf{h}_i\|_2^2.
    \end{equation}
    
    Unrolling Equation~\eqref{eq:orthogonal-split} gives $\|\mathbf{h}_i\|_2^2 =2^{-i}\|\mathbf{h}_0\|_2^2 +\sum_{j=0}^{i-1}2^{-(i-j)}\|\boldsymbol{\Delta}_j\|_2^2$. 
    Since $N_i=2^iN_0$, the coefficient of $\|\mathbf{h}_0\|_2^2$ in the last sum in Equation~\eqref{eq:tower-pre-sum} is bounded by
    \[
    \sum_{i=0}^{n}\frac d{N_i}2^{-i}
    \le\frac d{N_0}\sum_{i=0}^{\infty}4^{-i}
    =\frac{4d}{3N_0}.
    \]
    For each fixed $j<n$, the coefficient of $\|\boldsymbol{\Delta}_j\|_2^2$ in the last sum in Equation~\eqref{eq:tower-pre-sum} is bounded by
    \[
    \sum_{i=j+1}^{n}\frac d{N_i}2^{-(i-j)}
    \le\frac d{N_j}\sum_{t=1}^{\infty}4^{-t}
    =\frac13\frac d{N_j}.
    \]
    Consequently,
    \begin{equation}\label{eq:multiscale-sum}
        \sum_{i=0}^{n}\frac d{N_i}\|\mathbf{h}_i\|_2^2
        \le\frac{4d}{3N_0}\|\mathbf{h}_0\|_2^2
        +\frac13\sum_{i=0}^{n-1}\frac d{N_i}\|\boldsymbol{\Delta}_i\|_2^2.
    \end{equation}
    Substituting this into Equation~\eqref{eq:tower-pre-sum}, we obtain 
    \begin{align*}
        \Lambda_{P_n}(\mathbf{u}_{N_n})-\Lambda_{P_n}(\mathbf{x})
        \ge\left(\kappa_d-\frac{4B_d\,d}{3N_0}\right)\|\mathbf{h}_0\|_2^2
        +\left(\frac12-\frac{B_d}{3}\right)
        \sum_{i=0}^{n-1}\frac d{N_i}\|\boldsymbol{\Delta}_i\|_2^2.
    \end{align*}
    By Equation~\eqref{eq:N0-condition}, the first coefficient is at least $\kappa_d/2$.
    Since $d\ge160$, we have $\frac12-\frac{B_d}{3} =\frac1{20}-\frac4d \ge\frac1{40}$.
    This proves Equation~\eqref{eq:tower-stability}.
    If $\Lambda_{P_n}(\mathbf{u}_{N_n})=\Lambda_{P_n}(\mathbf{x})$, then Equation~\eqref{eq:tower-stability} forces $\mathbf{h}_0=\mathbf{0}$ and $\boldsymbol{\Delta}_i=\mathbf{0}$ for every $i<n$.
    Equation~\eqref{eq:orthogonal-split} then gives $\mathbf{h}_n=\mathbf{0}$, so $\mathbf{x}=\mathbf{x}^{(n)}=\mathbf{u}_{N_n}$.
    Thus the uniform weighting is the unique maximizer.
\end{proof}

%%%%%%%%%%%%%%%%%%%%%%%%%%%%%%%%%%%%%%%%%%%%%%%%%%%%
\section{The interval construction}
\label{sec:interval-construction}
%%%%%%%%%%%%%%%%%%%%%%%%%%%%%%%%%%%%%%%%%%%%%%%%%%%%

Theorem~\ref{thm:tower-stability} shows that the uniform weighting remains the unique maximizer throughout the palette tower, even when up to four suitable tripartite perfect matchings are deleted at each step.
We now choose these matchings recursively and use their deletion costs to produce an interval of limiting palette Lagrangians.
At each positive level $i$, deleting $k_i\in\{0,1,2,3,4\}$ designated matchings decreases the Lagrangian by exactly $k_i/N_i^2$.
Since $N_i=2^iN_0$, these costs are proportional to $4^{-i}$, and the choices of the integers $k_i$ give a redundant base-$4$ expansion.
This first produces a non-degenerate interval $[a,b]\subseteq\PiU$.
We then apply complete joins, which send a density $x$ to $1-(1-x)/M^2$ while preserving the homomorphisms between successive palettes.
For all sufficiently large consecutive values of $M$, the images of $[a,b]$ overlap and fill an interval ending at $1$.

We first fix the root size used throughout this section.
Proposition~\ref{prop:lift-tower} allows $N_\star$ to be chosen arbitrarily large, so we choose it such that the resulting $N_0$ satisfies Equation~\eqref{eq:N0-condition} and
\begin{equation}\label{eq:N0-later}
    N_0\ge7
    \quad\text{and}\quad
    N_0>\frac{2d+2}{3}.
\end{equation}
The condition $N_0\ge7$ ensures that the initial matchings constructed below use three pairwise distinct colors in every triple.
The second inequality will imply that the upper endpoint $b$ of the initial interval is strictly smaller than $1$.

\subsection{Persistent digit matchings}
\label{sec:digit-matchings}

It is not enough to find four deletable matchings at a single level: after making a choice there, we must still have four new matchings available at every later level.
We achieve this by starting with four matchings at the root and assigning each one a continuation reservoir.
Whenever the colors are doubled, the lifted copies of a continuation reservoir split into four new tripartite perfect matchings.
One of them becomes available for deletion at the new level, while another is retained to continue the construction.

Relabel the vertices of $D_0$ by $\mathbb Z/N_0\mathbb Z$.
For $j\in[4]$, define
\begin{equation}\label{eq:root-reservoirs}
    R_{0,j}\coloneqq \{(t,t+1,t+j+1)\colon t\in\mathbb Z/N_0\mathbb Z\}.
\end{equation}
For each fixed $j$, the three coordinate maps in Equation~\eqref{eq:root-reservoirs} are translations of $\mathbb Z/N_0\mathbb Z$ and hence are bijections, so $R_{0,j}$ is a tripartite perfect matching.
The three entries of each triple have pairwise differences $1$, $j$, and $j+1$, which are all nonzero modulo $N_0$ because $j\in[4]$ and $N_0\ge7$.
The matchings $R_{0,1},\ldots,R_{0,4}$ are pairwise disjoint: if a triple belonged to both $R_{0,j}$ and $R_{0,j'}$, then its first and third coordinates would imply $j=j'$.
Thus they are four pairwise disjoint matchings consisting entirely of triples of distinct colors, and all their triples are admissible in the root palette $P_0$.

Now suppose that $R_{i,j}$ has been retained as a continuation reservoir at level $i$.
Its triples have pairwise distinct colors and are admissible in $P_i$.
They are therefore unaffected by the repeated-color deletion in the next spectral palette lift.
For each $\alpha,\beta\in\{+1,-1\}$, define
\begin{equation}\label{eq:lifted-matchings}
    R_{i,j}^{\alpha,\beta}\coloneqq \{(a^\sigma,b^{\alpha\sigma},c^{\beta\sigma})\colon
    (a,b,c)\in R_{i,j},\ \sigma\in\{+1,-1\}\},
\end{equation}
which contains $2N_i=N_{i+1}$ triples.
Each coordinate map is a bijection onto the child color set: the corresponding coordinate map of $R_{i,j}$ is a bijection onto the parent color set, and multiplication by the fixed sign $\alpha$ or $\beta$ permutes $\{+1,-1\}$.
Hence every $R_{i,j}^{\alpha,\beta}$ is a tripartite perfect matching.
Furthermore, a sign triple $(\sigma_1,\sigma_2,\sigma_3)$ determines the parameters uniquely by $(\sigma, \alpha, \beta) = (\sigma_1, \sigma_2\sigma_1, \sigma_3\sigma_1)$. 
Consequently, the four matchings in Equation~\eqref{eq:lifted-matchings} are pairwise disjoint and together contain all lifted copies of the triples in $R_{i,j}$.
Matchings arising from different values of $j$ are also disjoint, since their projections to the parent color set lie in the disjoint reservoirs $R_{i,j}$.

For each $j\in[4]$, at level $i+1$ designate $M_{i+1,j}\coloneqq R_{i,j}^{+,+}$ and $R_{i+1,j}\coloneqq R_{i,j}^{+,-}$. 
We call $M_{i+1,j}$ the $j$th \emph{digit matching} at level $i+1$ and retain $R_{i+1,j}$ as the next \emph{continuation reservoir}.
The other two lifted matchings remain admissible but play no further role.
This recursive choice provides four fresh, pairwise disjoint digit matchings at every positive level.
Since the continuation reservoirs are never deleted, the four matchings available at a given level do not depend on any earlier deletion choices.

For a \emph{digit sequence} $\mathbf k=(k_1,k_2,\ldots)\in\{0,1,2,3,4\}^{\N}$, define a chain $P_n(\mathbf k)$ as follows.
Starting with $P_0(\mathbf k)=P_0$, construct $P_i(\mathbf k)$ from $P_{i-1}(\mathbf k)$ by performing the spectral palette lift and then deleting precisely $M_{i,1},\ldots,M_{i,k_i}$. 
If $k_i=0$, no digit matching is deleted at level $i$.
The continuation reservoirs are never deleted.
The deleted matchings satisfy the hypotheses of Theorem~\ref{thm:tower-stability}, and the natural projection gives a homomorphism $P_{n+1}(\mathbf k)\to P_n(\mathbf k)$ for every $n$.

The stability theorem now makes the Lagrangian of each finite palette in this chain an exact bookkeeping of the losses incurred at the successive steps.

\begin{proposition}\label{prop:branch-values}
    Let $L_0\coloneqq \lambda(P_0)=1-3/N_0+(2+3d)/N_0^2$.
    For every digit sequence $\mathbf k$,
    \begin{equation}\label{eq:finite-branch-value}
        \lambda(P_n(\mathbf k))
        =L_0-\sum_{i=0}^{n-1}\frac{3d}{4N_i^2}
        -\sum_{i=1}^{n}\frac{k_i}{N_i^2}.
    \end{equation}
    Consequently,
    \begin{equation}\label{eq:branch-limit}
        \lim_{n\to\infty}\lambda(P_n(\mathbf k))
        =L_0-\frac d{N_0^2}
        -\frac1{N_0^2}\sum_{i=1}^{\infty}k_i4^{-i}.
    \end{equation}
\end{proposition}

\begin{proof}
    Theorem~\ref{thm:tower-stability} gives $\lambda(P_n(\mathbf k)) = \Lambda_{P_n(\mathbf k)}(\mathbf{u}_{N_n})$ at every finite level.
    When we pass from level $i-1$ to level $i$, Equation~\eqref{eq:lift-uniform-loss} gives a loss of $3d/(4N_{i-1}^2)$ from the spectral palette lift, and the $k_i$ digit matchings deleted at that step contribute a further loss of $k_i/N_i^2$.
    Starting with $\lambda(P_0)=L_0$ and summing these losses proves Equation~\eqref{eq:finite-branch-value}.
    In particular,
    \[
    \lambda(P_i(\mathbf k))-\lambda(P_{i-1}(\mathbf k))
    =-\frac{3d}{4N_{i-1}^2}-\frac{k_i}{N_i^2}<0,
    \]
    so the Lagrangians along the chain are strictly decreasing.
    Finally, $N_i=2^iN_0$, and hence
    \[
    \sum_{i=0}^{\infty}\frac{3d}{4N_i^2}
    =\frac{3d}{4N_0^2}\sum_{i=0}^{\infty}4^{-i}
    =\frac d{N_0^2}.
    \]
    Taking the limit in Equation~\eqref{eq:finite-branch-value} gives Equation~\eqref{eq:branch-limit}.
\end{proof}

It remains to determine the set of values represented by the digit series in Equation~\eqref{eq:branch-limit}.
The extra digit $4$ makes the five intervals corresponding to the possible first digits overlap consecutively, which is why the series fills an interval rather than a disconnected set.

\begin{lemma}\label{lem:interval-coding}
    The set of possible digit sums is
    \[
    \left\{\sum_{i=1}^{\infty}k_i4^{-i}\colon
    k_i\in\{0,1,2,3,4\}\right\}
    =\left[0,\frac43\right].
    \]
\end{lemma}

\begin{proof}
    Since $0\le k_i\le4$, every digit sum lies in $[0,4/3]$.
    For the reverse inclusion, let $I=[0,4/3]$.
    The five intervals
    \[
    \frac{k+I}{4}
    =\left[\frac k4,\frac k4+\frac13\right],
    \qquad k=0,1,2,3,4,
    \]
    cover $I$ because each interval (except the first one) begins before the preceding one ends.
    Given $x_0\in I$, we may therefore choose $k_1\in\{0,\ldots,4\}$ and $x_1\in I$ such that $x_0=(k_1+x_1)/4$.
    Repeating this choice gives $k_i\in\{0,\ldots,4\}$ and $x_i\in I$ satisfying $x_{i-1}=(k_i+x_i)/4$ for every $i\ge1$.
    Iterating these identities yields
    \[
    x_0=\sum_{i=1}^{m}k_i4^{-i}+4^{-m}x_m.
    \]
    Since $0\le x_m\le4/3$, the final term tends to zero as $m\to\infty$, proving that $x_0$ has the required digit expansion.
\end{proof}

Combining the exact branch limits with the digit-coding lemma produces the initial interval of uniform Tur\'an densities.

\begin{theorem}\label{thm:base-interval}
    With $d$ and $N_0$ chosen as above, define
    \begin{equation}\label{eq:base-interval}
        a\coloneqq L_0-\frac{d+\frac{4}{3}}{N_0^2}
        \quad\text{and}\quad
        b\coloneqq L_0-\frac d{N_0^2}.
    \end{equation}
    Then $0<a<b<1$, and $[a,b]\subseteq\PiU$. 
\end{theorem}

\begin{proof}
    Substituting the value of $L_0$ from Equation~\eqref{eq:root-uniform} gives
    \[
    a=1-\frac3{N_0}+\frac{2d+\frac{2}{3}}{N_0^2}
    \quad\text{and}\quad
    b=1-\frac3{N_0}+\frac{2d+2}{N_0^2}.
    \]
    Since $N_0\ge7$, we have $a>1-3/N_0>0$.
    Moreover, $b-a=4/(3N_0^2)>0$, while the inequality $N_0>(2d+2)/3$ in Equation~\eqref{eq:N0-later} gives $b<1$.
    Thus $0<a<b<1$.
    
    Now let $x\in[a,b]$.
    By Lemma~\ref{lem:interval-coding}, we may choose a digit sequence $\mathbf k$ for which the limit in Equation~\eqref{eq:branch-limit} equals $x$.
    Proposition~\ref{prop:branch-values} shows that the corresponding Lagrangians $\lambda(P_n(\mathbf k))$ decrease strictly to $x$, and the natural projections make these palettes a projective chain.
    After reindexing the chain from $1$, Proposition~\ref{prop:projective-chain} gives $x\in\PiU$.
    Since $x$ was arbitrary, $[a,b]\subseteq\PiU$.
\end{proof}

%%%%%%%%%%%%%%%%%%%%%%%%%%%%%%%%%%%%%%%%%%%%%%%%%%%%
\subsection{Complete joins and the terminal interval}
\label{sec:complete-join}
%%%%%%%%%%%%%%%%%%%%%%%%%%%%%%%%%%%%%%%%%%%%%%%%%%%%

Theorem~\ref{thm:base-interval} gives a non-degenerate interval contained in $\PiU$, but this interval is bounded away from $1$.
We now move copies of this interval toward $1$ using complete joins.
The operation is compatible with palette homomorphisms, so it can be applied simultaneously to every palette in a projective chain.

Let $P=(C,T)$ be a palette and let $M\ge2$.
The \emph{$M$-fold complete join} $J_M(P)$ has color set $[M]\times C$, viewed as $M$ labeled copies of $C$.
An ordered triple $\left( (i,a),(j,b),(k,c) \right)$ is admissible whenever its copy indices $i,j,k$ are not all equal; if $i=j=k$, it is admissible precisely when $(a,b,c)\in T$.
Thus all triples using at least two different copies are admitted, while the restriction to each individual copy is the original palette $P$.

The next lemma records both the effect of this operation on the Lagrangian and its compatibility with palette homomorphisms.

\begin{lemma}\label{lem:join-formula}
    For every finite palette $P$ and integer $M\ge2$,
    \begin{equation}\label{eq:join-formula}
        \lambda(J_M(P))=1-\frac{1-\lambda(P)}{M^2}.
    \end{equation}
    Moreover, every homomorphism $P\to Q$ induces a homomorphism $J_M(P)\to J_M(Q)$.
\end{lemma}

\begin{proof}
    For a weighting $\mathbf{x}$ on $[M]\times C$, let $s_i$ be the total weight assigned to the $i$th copy of $C$, and write $\mathbf{s}=(s_1,\ldots,s_M)\in\Delta_M$.
    The total weight of all ordered triples is $1$.
    The triples whose three copy indices are equal have total weight $\sum_{i=1}^M s_i^3$, so the admissible triples using at least two different copies contribute $1-\sum_{i=1}^M s_i^3$. 
    Within the $i$th copy, the contribution is at most $\lambda(P)s_i^3$.
    For fixed $\mathbf{s}$, equality can be attained independently in every copy of positive weight by using a maximizing weighting for $P$; a copy of weight zero contributes nothing.
    Hence
    \[
    \lambda(J_M(P))
    =\max_{\mathbf{s}\in\Delta_M}
    \left(1-(1-\lambda(P))\sum_{i=1}^{M}s_i^3\right).
    \]
    Since $0\le\lambda(P)\le1$, this expression is maximized by minimizing $\sum_i s_i^3$ when $\lambda(P)<1$ and is identically $1$ when $\lambda(P)=1$.
    The inequality $\sum_{i=1}^M s_i^3\ge\frac1{M^2}$ holds for every $\mathbf{s}\in\Delta_M$, with equality at $\mathbf{u}_M$.
    This proves Equation~\eqref{eq:join-formula} in both cases.
    
    Finally, suppose that $f\colon P\to Q$ is a palette homomorphism.
    The map $(i,a)\mapsto(i,f(a))$ preserves the copy indices and applies $f$ within each copy.
    It therefore defines a homomorphism $J_M(P)\to J_M(Q)$.
\end{proof}

\begin{proof}[Proof of Theorem~\ref{thm:main}]
    Let $[a,b]\subseteq\PiU$ be the interval from Theorem~\ref{thm:base-interval}.
    Fix $M\ge2$ and $x\in[a,b]$.
    Choose the projective palette chain from the proof of Theorem~\ref{thm:base-interval} whose Lagrangians decrease to $x$, and apply $J_M$ to every palette in that chain.
    By Lemma~\ref{lem:join-formula}, the joined palettes still form a projective chain, and their Lagrangians decrease to $1-(1-x)/M^2$.
    Proposition~\ref{prop:projective-chain} therefore shows that $1-(1-x)/M^2\in\PiU$.
    As $x$ ranges over $[a,b]$, we obtain
    \[
    I_M\coloneqq
    \left[1-\frac{1-a}{M^2},\,1-\frac{1-b}{M^2}\right]
    \subseteq\PiU.
    \]
    
    Set $A\coloneqq1-a$ and $B\coloneqq1-b$.
    Then $A>B>0$ and $I_M=\left[1-\frac A{M^2},\,1-\frac B{M^2}\right]$. 
    Since $\left(\frac{M+1}{M}\right)^2\longrightarrow1<\frac AB$, there exists $M_0\ge2$ such that $\frac B{M^2}\le\frac A{(M+1)^2}$ for all $M\ge M_0$.
    Equivalently, the left endpoint of $I_{M+1}$ is at most the right endpoint of $I_M$.
    Hence the intervals $I_M$ for consecutive $M\ge M_0$ overlap.
    Their endpoints approach $1$, and therefore
    \[
    \bigcup_{M\ge M_0}I_M
    =\left[1-\frac A{M_0^2},1\right).
    \]
    Finally, $1=\pivvv(\varnothing)\in\PiU$.
    Thus $[1-\delta,1]\subseteq\PiU$ for $\delta\coloneqq\frac{1-a}{M_0^2}>0$, 
    which completes the proof.
\end{proof}

%%%%%%%%%%%%%%%%%%%%%%%%%%%%%%%%%%%%%%%%%%%%%%%%%%%%
\section{Concluding remarks}
\label{sec:concluding-remarks}
%%%%%%%%%%%%%%%%%%%%%%%%%%%%%%%%%%%%%%%%%%%%%%%%%%%%

We conclude by recording that the argument extends to every uniformity and, in
fact, produces one common terminal interval.
Fix an integer $r\ge3$.
For an $(r-2)$-graph $G$ on a vertex set $V$, put
\[
\mathcal K_r(G)
\coloneqq
\left\{e\in\tbinom{V}{r} \colon \tbinom e{r-2}\subseteq E(G)\right\}.
\]
An $r$-graph $H$ on $n$ vertices is \emph{$(d,\eta,r-2)$-dense} if
\[
|E(H)\cap\mathcal K_r(G)|
\ge d|\mathcal K_r(G)|-\eta n^r
\]
for every $(r-2)$-graph $G$ on $V(H)$.
For $r=3$, this is precisely the uniform density condition used throughout the
paper.
Let $\Pi_{\mathrm{u},\infty}^{(r)}$ be the set of the corresponding
$(r-2)$-uniform Tur\'an densities of possibly infinite families of $r$-graphs.

\begin{corollary}\label{cor:common-terminal-interval}
    The same constant $\delta>0$ obtained in the proof of Theorem~\ref{thm:main} satisfies
    \[
    [1-\delta,1]
    \subseteq
    \bigcap_{r\ge3}\Pi_{\mathrm{u},\infty}^{(r)}.
    \]
\end{corollary}

\begin{proof}
Only a transfer argument is needed.
In the natural $r$-palette formalism, colors are assigned to $(r-1)$-sets by a map $\phi\colon\binom{V}{r-1}\to C$, and an admissible set $T'\subseteq C^r$ prescribes the ordered face-color patterns of the edges.
For an edge $e=\{v_1,\ldots,v_r\}$ with $v_1<\cdots<v_r$, we record these colors in the order
\[
\bigl(\phi(e\setminus\{v_r\}),\phi(e\setminus\{v_{r-1}\}),\ldots,
\phi(e\setminus\{v_1\})\bigr).
\]
This is the reverse of the coordinate order in~\cite{LinSunWangZhou2026}, so for $r=3$ it agrees with the convention in Section~\ref{se:palettes}; reversing all coordinates translates their results into the convention used here.
For an $r$-palette $P'=(C,T')$ and a weighting $\mathbf{x}\in\Delta_C$, define
\[
\Lambda_{P'}(\mathbf{x})
\coloneqq\sum_{(c_1,\ldots,c_r)\in T'}\prod_{j=1}^r x_{c_j},
\quad\text{and}\quad 
\lambda(P')\coloneqq\max_{\mathbf{x}\in\Delta_C}
\Lambda_{P'}(\mathbf{x}).
\]
Its unweighted density is $d(P')\coloneqq |T'|/|C|^r=\Lambda_{P'}(\mathbf{u}_C)$.
Given a $3$-palette $P=(C,T)$, define its \emph{cylindrical extension} to uniformity $r$ by $\operatorname{Cyl}_{3\to r}(P) \coloneqq \bigl(C,T\times C^{r-3}\bigr)$. 
Thus the first three coordinates are constrained by $P$, while the remaining $r-3$ coordinates are unrestricted.
For every $\mathbf{x}\in\Delta_C$, the corresponding palette polynomial
satisfies
\[
\Lambda_{\operatorname{Cyl}_{3\to r}(P)}(\mathbf{x})
=
\Lambda_P(\mathbf{x})
\Big(\sum_{c\in C}x_c\Big)^{r-3}
=
\Lambda_P(\mathbf{x}).
\]
It follows, by evaluating at the uniform weighting and by maximizing over $\Delta_C$, that cylindrical extension preserves the unweighted density and the Lagrangian.
Moreover, every homomorphism $P\to Q$ induces a homomorphism $\operatorname{Cyl}_{3\to r}(P) \longrightarrow \operatorname{Cyl}_{3\to r}(Q)$ by applying the same map in every coordinate.

Consequently, if $P_1,P_2,\ldots$ is a projective chain of $3$-palettes, then the palettes $\operatorname{Cyl}_{3\to r}(P_1), \operatorname{Cyl}_{3\to r}(P_2),\ldots$ form a projective chain of $r$-palettes with the same Lagrangians.
The weighted random-palette construction in Lemma~\ref{lem:weighted-palette} extends verbatim to $r$-palettes, using independent colors on $(r-1)$-sets.
Together with the palette separation and arbitrary-family characterization theorems of Lin, Sun, Wang and Zhou~\cite[Theorems~1.4 and~1.6]{LinSunWangZhou2026}, this shows that the proof of Proposition~\ref{prop:projective-chain} applies unchanged to projective chains of $r$-palettes.
It follows that every decreasing Lagrangian limit constructed above belongs to $\Pi_{\mathrm{u},\infty}^{(r)}$ for every $r\ge3$.
In particular, the interval $[a,b]$ from
Theorem~\ref{thm:base-interval} occurs simultaneously in all uniformities.

To obtain a common terminal interval, we first apply the complete join at the $3$-palette level and only then take the cylindrical extension.
If $\lambda(P_n)\downarrow x$, then Lemma~\ref{lem:join-formula} gives $\lambda(J_M(P_n)) \downarrow 1-\frac{1-x}{M^2}$ and cylindrical extension leaves this limit unchanged. Thus every interval $I_M$ from the proof of Theorem~\ref{thm:main} is contained in $\Pi_{\mathrm{u},\infty}^{(r)}$ for every $r\ge3$.
The overlap argument depends only on $a$ and $b$, so the same $M_0$ and the same value $\delta=(1-a)/M_0^2$ work in every uniformity.
Performing the join before the cylindrical extension is what keeps the exponent $2$ in the join formula, and hence keeps the resulting terminal interval independent of $r$.
\end{proof}

%%%%%%%%%%%%%%%%%%%%%%%%%%%%%%%%%%%%%%%%%%%%%%%%%%%%
\section*{Acknowledgments}

H.L. was supported by the National Natural Science Foundation of China (12501487), by the China Scholarship Council, and by the Institute for Basic Science (IBS-R029-C4).  X.L. was supported by the Excellent Young Talents Program (Overseas) of the National Natural Science Foundation of China. O.P. was supported by ERC Advanced Grant 101020255.

\section*{Declaration on the use of AI}

The authors used ChatGPT 5.6 Pro to assist in discussing proof strategies, checking proofs, and improving exposition.

%%%%%%%%%%%%%%%%%%%%%%%%%%%%%%%%%%%%%%%%%%%%%%%%%%%%
\bibliographystyle{abbrv}
\bibliography{uniform}
%%%%%%%%%%%%%%%%%%%%%%%%%%%%%%%%%%%%%%%%%%%%%%%%%%%%

\end{document}